\documentclass[11pt]{amsart}
\usepackage{amscd}
\usepackage{amsmath}
\usepackage{amsthm}
\usepackage{amsfonts}
\usepackage{amssymb}
\usepackage{color}

\numberwithin{equation}{section}
\usepackage{hyperref}

\usepackage{tikz-cd}

\newtheorem{pro}{Proposition}[section]
\newtheorem{proposition}[pro]{Proposition}
\newtheorem{lemma}[pro]{Lemma}
\newtheorem{theorem}[pro]{Theorem}
\newtheorem{thm:intro}{Theorem}
\newtheorem{cor:intro}{Corollary}
\newtheorem{corollary}[pro]{Corollary}
\newtheorem*{corollary*}{Corollary}
\newtheorem*{thm*}{Theorem}

{\theoremstyle{definition}
\newtheorem{definition}[pro]{Definition}

\newtheorem{example}[pro]{Example}

 }

{\theoremstyle{remark}
\newtheorem*{remark}{Remark}
}

\newcommand{\ra}{\rightarrow}
\newcommand{\lra}{\longrightarrow}

\newcommand{\EE}{{\mathbb E}}

\newcommand{\HH}{{\mathbb H}}

\newcommand{\RR}{{\mathbb R}}

\newcommand{\ZZ}{{\mathbb Z}}

\newcommand{\GL}{\mathop{\rm GL}\nolimits}
\newcommand{\SL}{\mathop{\rm SL}\nolimits}
\newcommand{\PSL}{\mathop{\rm PSL}\nolimits}
\newcommand{\PSO}{\mathop{\rm PSO}\nolimits}
\newcommand{\SO}{\mathop{\rm SO}\nolimits}

\newcommand{\Hom}{\mathop{\rm Hom}\nolimits}

\newcommand{\Aut}{\mathop{\rm Aut}\nolimits}
\newcommand{\Out}{\mathop{\rm Out}\nolimits}

\newcommand{\rank}{\mathop{\rm rank}\nolimits}

\newcommand{\Diff}{\mathop{\rm Diff}\nolimits}
\newcommand{\Aff}{\mathop{\rm Aff}\nolimits}
\newcommand{\Isom}{\mathop{\rm Isom}\nolimits}

\newcommand{\SLt}{{\SL(2,\bbR)}}

\newcommand{\tSL}{\widetilde{\SL}(2,\bbR)}
\newcommand{\tSO}{\widetilde{\SO}(2,\bbR)}

\newcommand{\lie}[1]{{\mathfrak{#1}}}

\DeclareMathOperator{\rad}{{rad}}   

\newcommand{\cK}{{\mathcal{K}}}

\newcommand{\cZ}{{\mathcal{Z}}}

\newcommand{\fD}{{\mathsf{D}}}

\newcommand{\fF}{{\mathsf{F}}}
\newcommand{\fG}{{\mathsf{G}}}

\newcommand{\fK}{{\mathsf{K}}}

\newcommand{\fN}{{\mathsf{N}}}

\newcommand{\fR}{{\mathsf{R}}}
\newcommand{\fS}{{\mathsf{S}}}
\newcommand{\fT}{{\mathsf{T}}}

\newcommand{\bb}[1]{{\mathbb #1}}    
\newcommand{\bbR}{{\bb R}}

\newcommand{\bbZ}{{\bb Z}}

\newcommand{\ac}[1]{\overline{#1}}
\newcommand{\cf}{{cf$.$\,}}

\begin{document}

\pagestyle{myheadings} \thispagestyle{empty} \setcounter{page}{1}


\title[]
{Isometry groups with radical, and aspherical Riemannian manifolds with large symmetry I}
\author[]{Oliver Baues}
\address{Department of Mathematics\\ 
University of Fribourg\\
Chemin du Mus\' ee 23\\
CH-1700 Fribourg, Switzerland}
\email{oliver.baues@unifr.ch}

\author[]{Yoshinobu Kamishima}
\address{Department of Mathematics, Josai University\\
Keyaki-dai 1-1, Sakado, Saitama 350-0295, Japan}
\email{kami@tmu.ac.jp}

\date{September 22, 2018} 

\begin{abstract}
Every compact \emph{aspherical} Riemannian manifold  admits a
canonical series of orbibundle structures with infrasolv fibers
which is called its \emph{infrasolv tower}. The tower arises from
the solvable radicals of isometry group actions on the universal covers.  
Its length and the geometry of its base measure the degree of continuous symmetry of an
aspherical Riemannian manifold. We say that the manifold has 
\emph{large symmetry} if it admits an infrasolv tower 
whose base is a locally homogeneous space.  
We construct examples of aspherical manifolds with large symmetry, 
which do not support any locally homogeneous Riemannian metrics. 
\end{abstract}

\maketitle


\section{Introduction}

Let $X$ be a contractible Riemannian manifold, and 
suppose there exists a discrete group $\Gamma$ of isometries of $X$ 
such that the quotient space $$ X \, \big/ \, \Gamma $$  is compact.
Then $X$ is said to be \emph{divisible}. 
If in addition $\Gamma$ is torsion-free,  the quotient space 
is a compact aspherical manifold. 
Let $$ {\rm Isom}(X)$$  denote the group of isometries of $X$. 
The action of the continuous part  
$$   \fG \, = \, \Isom(X)^0 $$ 
of the isometry group on the space $X$ may be seen as describing the \emph{local symmetry} of the 
quotient metric space $X  \big/ \, \Gamma$, which is a  Riemannian orbifold. 

\smallskip 
Under the assumption that $X  \big/ \, \Gamma$ is a manifold, 
Farb and Weinberger \cite{FW} proved  the fundamental fact 
that 
$$     X \,   \big/   \Isom(X)^0 $$ 
is again a contractible manifold, 
and is giving rise to an \emph{Riemannian orbibundle} 
$$   X \,   \big/ \, \Gamma  \; \to \;     X  \,   \big/ \,  \Gamma \Isom(X)^0  \; , $$ 
where the fibers are locally homogeneous spaces modelled on  
$$    X_{\fG}  \; =  \;    \fG   \, \big/   \, \fK  \; , $$ 
with $\fK$ a maximal compact subgroup of $\fG$.  

\smallskip 
In this setup, a basic observation 
is the fact, 
that  
the continuous part of the isometry group of the Riemannian quotient
 $$  X  /   \Isom(X)^0$$   may be non-trivial, 
and thereby reveals \emph{``hidden''}  local symmetries of the space $X$ 
and its quotient $X /  \Gamma$. 
From this point of view, the 
local symmetry of  $X /  \Gamma$ will be 
encoded  in an ensuing tower of orbifold fibrations with locally homogeneous fibers.
And we can say that $X /  \Gamma$ has \emph{``maximal''} local symmetry  if the 
tower finally stops  over a locally homogeneous orbifold.  

\smallskip 
The purpose of this article is to initiate 
the study of the \emph{local symmetry of aspherical spaces}  
in terms of their naturally  
associated {\em towers of Riemannian orbifold fibrations}. 

\medskip 
\subsection{The radical quotient and divisibility} 
As a starting point, 
we are concerned with the action of the maximal 
connected normal solvable subgroup
$$    \fR  \;  \leq \, \fG = \Isom(X)^{0} \;  , $$  
which is called the \emph{solvable radical} of $\fG $. We then call the quotient space 
$$   Y \, = \, X \big/ \, \fR $$   the \emph{radical quotient} of $X$. 

\medskip 
\paragraph{\em Divisibility of the radical quotient} 
As our  first main result we show that the radical quotient is a contractible Riemannian manifold, and that it  is 
divisible by the image of $\Gamma$ in  $ \Isom(X \big/ \, \fR)   $. 


\begin{thm:intro} [\cf Theorem $\ref{thm:Theta_dis}$] \label1
The radical quotient $X \big/ \, \fR$ is a contractible Riemannian manifold, 
and  the image $\Theta$ of\/ 
$\Gamma$ in\/  $\Isom(  X \big/ \, \fR )$ 
acts properly discontinuously on $X \big/ \, \fR $ 
with compact quotient. 
\end{thm:intro} 

\begin{remark} We may view this result  as a parametrised version of 
the classical Bieberbach Auslander Wang Theorem (\cf Proposition \ref{prop:key-correct}), 
which concerns lattices in Lie groups. 
\end{remark}

Note that Theorem 1 shows that the radical quotient gives rise to an \emph{associated  Riemannian orbibundle} of the form 
\begin{equation} \label{eq:rad_bundle}       X \,   \big/ \, \Gamma  \; \to \;     Y  \,   \big/ \, \Theta  \;  .   \end{equation}
In the following, we will study the geometry of this fibration in more detail. 

%

\smallskip 
Before doing so, let us mention
the following topological result, which is of independent
interest, and  is used as a basic tool for the proof of Theorem~\ref{1}.  
It asserts that there are no compact Lie group actions
normalized by properly discontinuous actions on \emph{acyclic smooth manifolds}
with compact quotient: 

\begin{thm:intro} [\cf Theorem $\ref {thm:compact_trivial}$]\label 2  
Let $X$ be an orientable acyclic manifold and $\Gamma$  a group which
acts smoothly and properly discontinuously on $X$ with compact quotient. Let $\kappa$ be a compact Lie group acting faithfully and smoothly on $X$, such that the action is normalized 
by\/ $\Gamma$. Then $\kappa =\{ 1 \}$. 
\end{thm:intro}

The proof of Theorem 2 is based  mainly on \emph{cohomology of groups acting on acyclic complexes and
application of the Smith theorem}. 

\smallskip 
The first geometric application of Theorem 2 is
the following.

\begin{cor:intro}[\cf Corollary \ref{cor:no_normal_compact}]\label{c1}
Let $X$ be a contractible Riemannian manifold which 
is divisible. Then $\Isom(X)$ has no non-trivial compact normal subgroup.   
 \end{cor:intro}

\begin{remark} A similar result was obtained  in \cite[Claim II]{FW} for connected compact normal subgroups 
of $\Isom(X)$ under the possibly stronger assumption that there exist $\Gamma$ such that $X / \Gamma$ is a manifold.
By the above, this generalizes to merely divisible Riemannian manifolds $X$, and Corollary \ref{c1}  simultaneously strengthens the statement  to include possibly non-connected 
groups.
\end{remark} 

Based on (the proof of) Theorem \ref{1} and Theorem \ref{2}, we carry out 
a detailed analysis for the interaction of the smooth properly discontinuous action of 
the discrete discrete group  $\Gamma$ on $X$ with the radicals of $\Isom(X)^{0}$. 
Some of the additional results, which we obtain, are summarized in the following structure theorem: 

\begin{thm:intro}[\cf Theorem $\ref {thm:structure}$]\label3
Let $X$ be a contractible Riemannian
manifold and $\Gamma \leq \Isom(X)$ a discrete subgroup 
such that $X/\Gamma$ is compact.  Let $\fR$ denote the solvable radical of\/ $\Isom(X)^0$. 
Then there exists a unique Riemannian metric on the quotient $X/\fR$ such that  the 
map $$X \ra X/\fR$$  is a Riemannian submersion. 
It follows further that:
\begin{enumerate}
\item $\Isom(X)/\fR$ acts properly on $X/\fR$. 
\item  The kernel of the action $(1)$ is  the maximal compact normal subgroup of\/ 
$\Isom(X) /\fR \, . $
\item The image  
of\/  $\Isom(X)^0$ in $\Isom(X/\fR)$ is a 
semisimple Lie group $S$ of non-compact type without finite subgroups in its center. 
Moreover, it is a closed  normal subgroup of\/ $ \Isom(X/\fR)^{0}$. (And therefore normal in a
finite index subgroup of\/ $\Isom(X/\fR)$.) 
\item 
Moreover, $\Theta \cap S$ is a uniform lattice in $S$. 
\end{enumerate}
\end{thm:intro}
Remark that the Lie group $S$ in (3) may have infinite (but discrete) center, and such examples 
do naturally occur (see Section \ref{sect:exsl2R}).  

\medskip 
\paragraph{\em Infrasolv orbifolds and orbibundles} 
Let $R$ be a simply connected solvable Lie group. We let
$$      \Aff(R) = R \rtimes \Aut(R) $$
denote its group of affine transformations. 
Consider  a discrete subgroup $$ \Delta \, \leq \, \Aff(R) \; , $$  such that the
homomorphic image of $\Delta$ in $\Aut(R)$ has compact closure.  We may choose 
some left-invariant Riemannian metric on $R$ such that 
$\Delta$ acts by isometries. Then the quotient space $$   R \, \big/ \Delta$$  
is an aspherical Riemannian 
orbifold. Compact orbifolds of this type are traditionally called \emph{infrasolv orbifolds}. 

\smallskip
\begin{remark}
Infrasolv \emph{manifolds} form an important class of aspherical 
locally homogeneous Riemannian manifolds, and play a 
role in various geometrical contexts.  
See \cite{Baues, Tuschmann, Wilking} for review 
on infrasolv manifolds. 
 \end{remark}

\smallskip 
A Riemannian orbibundle  
will be called an  \emph{infrasolv bundle} if its fibers 
are compact infrasolv orbifolds (with respect to  the induced metric) which 
are all modeled on the same Lie group $R$  (\cf Definition \ref{fibersolv}).

\smallskip 
We can then state our main result on the geometry  of
the Riemannian orbibundle \eqref{eq:rad_bundle} 
which is associated to the radical quotient: 

\begin{thm:intro} [\cf Theorem $\ref {thm:basic_bundle}$] \label 4
There exists a simply connected solvable normal Lie sugroup $R$ of $\fR$, such that 
$X/\Gamma$ has an induced structure of 
Riemannian infrasolv fiber space (modeled on $R$) over the compact aspherical
Riemannian orbifold $Y/\Theta$.
\end{thm:intro} 

\subsection{Infrasolv towers and  manifolds of large symmetry}
According to  Theorem \ref{4}, the radical quotient $Y = X / \fR$ is a contractible Riemannian manifold
and it is divisible by the image of $\Gamma$ in $\Isom(Y)$. 
In general, the isometry group 
$\Isom(Y)$ will 
have a non-trivial continous part $\Isom(Y)^{0}$, and also the radical of  $\Isom(Y)^{0}$
can be non-trivial. (We will discuss several types of  examples below.)

We may thus repeat the process of taking radical quotients until the continuous  part of the 
isometry group is semisimple (or trivial). This gives rise to a canonical tower of Riemannian orbibundles, 
which by Theorem \ref 4 are infrasolv orbibundles,  and are  filtering the space $X/\Gamma$. Such a tower will be  called an \emph{infrasolv tower} for $X/\Gamma$:

\begin{cor:intro}[\cf Corollary \ref {cor:solvtower}]\hspace{1ex}  \label {c2}
Every aspherical Riemannian orbifold $X/\Gamma$ gives rise to a 
canonical infrasolv tower 
\begin{equation} \label{eq:tower2i} 
X/\Gamma \lra  X_1/ \Gamma_1 \lra \,  \cdots \, \lra X_{\ell}/\Gamma_\ell  \; ,  
\end{equation} where 
the solvable radical of\/  $ \Isom(X_{\ell})$ is trivial.
\end{cor:intro}


\smallskip 
\paragraph{\em Riemannian manifolds of large symmetry}
The length $\ell$ and structure of the tower \eqref{eq:tower2i} describe canonical 
invariants of Riemannian metrics on $X$ and $X/\Gamma$. The following notion 
is thus generalising locally homogeneous Riemannian manifolds: 

\begin{definition}\label{def:symmetry} 
We say that the Riemannian orbifold $X/ \Gamma$ has \emph{large local symmetry}
 if,   in \eqref{eq:tower2i}, $X_{\ell}=\{{\rm pt}\}$ or $X_\ell$ is a Riemannian homogeneous manifold of a semisimple Lie group.
\end{definition}

We illustrate the concept by constructing in Section \ref{sect:examples} 
 several examples of aspherical manifolds with large symmetry. 
 A simple example is a $2$-dimensional warped product of circles  
$$ M_{f}= S^1 {\times}_{\bar f} \,  S^1 \; ,  $$  
for certain warping function $\bar f$. Here $M_{f}$ is diffeomorphic to a $2$-torus whose universal cover 
$\displaystyle X$ satisfies $\Isom(X)^{0}=\RR$ 
(\cf Example \ref {ex:revolution_torus}).
More generally,  in this direction, we may also 
take \emph{any\/} (compact) locally symmetric space of noncompact type 
$$\displaystyle  N =  \Theta \backslash S/K \,  , $$ and we can form the warped product
(with fiber $N$ and base $S^{1}$) 
 $$ M_{N,f} =  S^{1} {\times}_{\bar f} \, N \to  S^{1} \; , $$ 
to obtain Riemannian manifolds of 
large symmetry which are not locally homogeneous. 

\smallskip 
The above kind of examples for metrics of large symmetry are built on spaces 
which admit locally homogeneous metrics from the beginning, and therefore 
may be seen as  ``merely'' exhibiting symmetry properties of particular Riemannian metrics 
on these spaces. In the following we come to  discuss the \emph{topological significance} of the concept
of large symmetry. 

\subsection{Riemannian  orbibundles arising from group extensions} 
We introduce now a 
general method to construct infrasolv bundles, 
starting from abstract group extensions  of the form 
$$ 1 \to \Lambda \to  \Gamma \to \Theta \ra 1 \, , $$ 
where the group  $\Lambda$, in general,  will be a virtually polycyclic group, 
and $\Theta$ is a discrete group which is dividing a contractible Riemannian 
manifold $Y$. 

\smallskip 
The basic construction is partially based on the notion 
of \emph{injective Seifert fiber spaces} (as developed in \cite{LeeRaymond}). 
The details will be explained in Section \ref{sect:group_extensions}. As an 
application of this method, we can derive: 

\begin{thm:intro}[\cf Theorem $\ref{thm:non_locally_homogeneous}$]\label8
Let $\Theta$ be a torsion free uniform lattice in the hyperbolic group ${\rm PSO}(n,1)$. 
Suppose
$$ \displaystyle 1\ra \ZZ^k\ra \Gamma \to \Theta\ra 1$$
is  a central group extension 
which has infinite order. 
Then  there exists a compact aspherical manifold $X/\Gamma$ such that 
\begin{enumerate}
\item[(1)] $X/\Gamma$  admits a metric of large symmetry with length $\ell = 1$.
\item[(2)] For $n\geq 3$, $X/\Gamma$ does not admit a locally homogeneous Riemannian metric. 
In particular, $\Gamma$ does not embed as a uniform lattice into a connected Lie group.
\item[(3)]  For $n=2$, $X/\Gamma$ admits the structure of a locally homogeneous Riemannian manifold. In particular, $\Gamma$ embeds as a uniform lattice into a connected Lie group. 
\end{enumerate}
\end{thm:intro}

The space $X/\Gamma$, also called
a Seifert fibering over the compact hyperbolic manifold $\HH^n/\Theta$, 
will inherit a Riemannian orbifold bundle structure over $\HH^n/\Theta$
with typical fiber a $k$-torus and exceptional fiber 
a Euclidean space form.  
This construction can be applied to any compact
locally homogeneous symmetric manifold $\Theta\backslash G/K$
of real rank $1$ such that $H^2(\Theta,\ZZ^k)$ has an element of
infinite order.

\begin{cor:intro} [\cf Corollary $\ref{cor:non_locally_homogeneous}$]\label {c3}
There exists a compact aspherical Riemannian manifold $X/\Gamma$ of dimension four that 
admits a complete infrasolv tower of length one, 
which is fibering over a three-dimensional hyperbolic manifold.  
Moreover, the manifold $X/\Gamma$ does not admit any locally homogeneous 
Riemannian metric.
\end{cor:intro}

The corollary shows the topological significance of the concept of large symmetry: 
\emph{The class of aspherical smooth manifolds admitting a metric of large symmetry 
is strictly larger than the class of manifolds which can be presented as a Riemannian 
locally homogeneous space.}   

\subsection{Smooth rigidity problem for manifolds of large symmetry} 
In a closing result for this paper we briefly touch on a particular differential 
topological aspect of our topic. Namely, in the following we are concerned 
with the existence of metrics of large symmetry on smooth manifolds 
\emph{homeomorphic} to the torus. 

\smallskip 
An $n$-dimensional \emph{exotic torus} is a compact smooth manifold
homeomorphic to the standard $n$-torus $T^n$ but not diffeomorphic
to $T^n$. Such manifolds are known to exist by \cite{WA},
\cite{HS} for example.
We prove that
every smooth manifold homeomorphic to the torus which is 
admitting a metric of large symmetry must be diffeomorphic 
to the standard torus $T^n$. In other words: 

\begin{thm:intro}[\cf Theorem $\ref{theorem:fake}$] \label{6}
Let $\tau$ be an $n$-dimensional exotic torus. Then $\tau$ does not
admit any Riemannian metric of \emph{large symmetry}.
\end{thm:intro}

This result  is embedded in a wider context of smooth rigidity results 
which are known for certain classes of locally homogeneous manifolds
(or orbifolds), for example, locally symmetric spaces (by Mostow strong rigidity, \cite{Mostow1}) 
or infrasolv manifolds (as in \cite{Baues}). In fact, since such spaces constitute the 
building blocks of manifolds of large symmetry, we would like to pose the
following 

 \smallskip  \noindent
 \emph{Question: Is the class of aspherical smooth manifolds admitting a 
metric of large symmetry smoothly rigid $?$ That is, given any two aspherical 
manifolds of large symmetry with isomorphic fundamental group, are they
diffeomorphic?}

\bigskip 
\paragraph{\em Organization of the paper} 

The paper is organized as follows.  
In Section \ref{sec:prel} we briefly discuss proper actions of Lie groups, and properly discontinuous actions
of discrete subgroups.
In Section \ref{sect:smooth_crystallographic} we prove Theorem \ref{2} and Corollary \ref{1}.
Topological concepts such as group
cohomology for infinite discrete groups and the Smith theorem 
for acyclic manifolds play a role here.
In Section \ref{sect:Isomradical} we study the action of 
the solvable radical of the isometry group
$\Isom(X)$ for divisible Riemannian manifolds $X$, and  
we prove Theorem \ref{1}. 
Also the structure theorem, Theorem \ref{3}, is proved here. 
In Section \ref{sec:soltower} we introduce the notion of infrasolv tower and
the concept of large symmetry to prove our corner stone results
Theorem \ref{4} and Corollary \ref{c2}. In Section \ref{sect:examples}  and
Section \ref{sect:group_extensions} we give several examples of aspherical manifolds 
with large symmetry, as well as a general construction method 
related to Seifert fiber spaces, which
gives Theorem \ref{8}, Corollary \ref{c3}.
In Section \ref{sec:fake} we prove the smooth rigidity theorem, Theorem 
\ref{6},  for topological tori with large symmetry. 

\bigskip 
\paragraph{\bf Acknowledgement}
\thanks{This research was performed
during the first author's stay at Tokyo Metropolitan University and Josai university
in 2017, 2018. He gratefully acknowledges the support of the Japanese
Science Foundation grants { no.15K04852, no.18K03284} during his stay.
The second author acknowledges the support of Mathematisches Institut
G\"ottingen during his stay in 2015.} 

\section{Preliminaries}\label{sec:prel}
\subsection{Solvable Lie groups}
Let $R$ be a connected solvable Lie group and $T \leq R$
a maximal compact subgroup. Let $N$ be the nilpotent radical of 
$R$, that is, its maximal connected normal nilpotent subgroup.
The intersection $N \cap T$ is the maximal compact subgroup
of $N$, which is central and characteristic 
in $N$, since $N$ is nilpotent. 
\begin{lemma}\label{lem:chara-R0}
Assume that $N$ is simply connected.
Then there exists a characteristic simply connected 
subgroup $R_0$ of $R$, with $N \leq R_{0}$, such
that $R= R_0 \rtimes T$.
\end{lemma}
\begin{proof} Dividing by $N$, we obtain a quotient map 
 $\displaystyle  R\stackrel{p}\to  R/N = T_1\times V_2$,
where $T_1=p(T)$ and $V_{2}$ is a vector group. Let $V= V_{1} \times V_2 $ be the universal covering group 
of $T_1\times V_2$, where $V_{1}$ covers $T_{1}$.
Let $\tilde R$ be the universal covering 
of $R$. Then $V = \tilde R/N$. Under the induced 
action of  the automorphism group ${\rm Aut}(\tilde R)$ on
$V = \tilde R/N$ the identity component ${\rm Aut}(\tilde R)^0$ acts trivially  
(see \cite[Ch.\ III, Theorem 7]{Jacobson}). 
As ${\rm Aut}(\tilde R)$ is an algebraic group, 
${\rm Aut}(\tilde R)/{\rm Aut}(\tilde R)^0$ is finite. 
Note that $\Aut(R)$ is a subgroup of ${\rm Aut}(\tilde R)$. 
In particular, the image of $\Aut(R)$ in 
$\GL(V)$ is finite. As the factor $V_{1}$ is invariant under
$\Aut(R)$, there is an invariant complementary subspace in $V$, 
which projects to an $\Aut(R)$-invariant 
vector subgroup $V_{2}'$ of $T_1\times V_2$.
Then $R_0=p^{-1}(V_2')$ is a subgroup of 
$R$,  which is invariant under ${\rm Aut}(R)$.  
It is obvious that $R_0$ is
simply connected and that $R= R_0\rtimes \fT$.
\end{proof}

\subsection{Discrete subgroups of Lie groups} 
Let $G$ be a Lie group.
We call a closed subgroup $H$ of $G$ uniform if $G/H$ is compact. A discrete uniform subgroup is called a uniform lattice. 
We note the following elementary fact on uniform subgroups
(compare \cite[Theorem 1.13]{Raghunathan}):

\begin{lemma} \label{lem:uniformsubgs}
Let $H$ be a uniform subgroup of $G$
and $L$ a closed (normal) subgroup of $G$.  If $L/ L\cap H$ is compact then $L H$ is closed in $G$. Moreover, 
if $H= \Gamma$ is discrete then $L/ L\cap \Gamma$
is compact if and only if $L \, \Gamma$ is closed in $G$. 
\end{lemma}
\begin{proof} 
To prove the first claim, we consider the map $L/ L\cap H \ra G/H$. Now if $\Gamma$ is discrete the map  $L/L \cap \Gamma \ra L \Gamma/ \Gamma \subset G/\Gamma$  is a homeomorphism. This implies the second part of the lemma.
\end{proof}

Here is a simple application:

\begin{lemma} \label{lem:lattice_in_component}
Let $\Gamma$ be a uniform subgroup of $G$. Let $G^{\, 0}$ denote the identity component of $G$. 
Then $\Gamma \cap G^{\, 0}$ is a uniform lattice in $G^{\, 0}$. 
\end{lemma} 
\begin{proof} Since $ G/ G^{\, 0}$ is discrete, the image of $\Gamma$ in $ G/ G^{\, 0}$ is discrete. 
Therefore $\Gamma G^{\, 0}$ is closed in $G$. By Lemma \ref{lem:uniformsubgs}, $\Gamma \cap G^{\, 0}$ 
is a uniform lattice of $G^{\, 0}$. 
 \end{proof} 

\begin{lemma} \label{lem:uniformsubgs2}
Let $\Gamma$ be a uniform lattice of $G$,
$L$ a closed normal subgroup of $G$, and 
$\nu: G \ra G/L$ the quotient homomorphism.  
Assume that  the identity component $\ac{\nu(\Gamma)}^{\, 0}$
of the closure of $\nu(\Gamma)$ is contained in a
compact normal subgroup $K$ of $G/L$. 
Then $\Gamma \cap \nu^{-1}(K)$ is a uniform lattice in
$ \nu^{-1}(K)$.
\end{lemma}
\begin{proof}  By Lemma \ref{lem:uniformsubgs}, it is
enough to show that $\nu^{-1}(K) \Gamma$ is closed 
in $G$. By assumption,  $T= \ac{\nu(\Gamma)}^{\, 0}$ is a subgroup
of $K$.  As 
${\ac{\nu(\Gamma)}} = T \nu(\Gamma)$, 
we observe that 
$K \nu(\Gamma) = K  \ac{\nu(\Gamma)}$. 
Since $K$ is compact, so is $K/ K \cap  \ac{\nu(\Gamma)}$.
Clearly, $\ac{\nu(\Gamma)}$ is a uniform
subgroup of $G/L$. Hence, by  the first part of Lemma \ref{lem:uniformsubgs}, 
$K \ac{ \nu(\Gamma)}$ is closed in $G/L$. 
Therefore, $\nu^{-1}(K) \Gamma = \nu^{-1}(K \nu(\Gamma))$ is closed 
in $G$.
 \end{proof}

\subsubsection{Levi decomposition} \label{sect:Levi_dec}
We will assume now that  $G$ is a connected Lie group. 
Let $\fR$ be the
solvable radical of $G$  (that is, $\fR$ is the maximal 
connected normal solvable subgroup of $G$). 
Then $G$ admits a Levi
decomposition $$ G= \fR\cdot \fS \; .$$ 
The Levi subgroup $\fS$ is a semisimple (not necessarily closed) immersed subgroup of $G$ 
and the intersection $\fR \cap \fS$ is a totally disconnected subgroup of $\fR$ and central in $\fS$.
We let  $\fK$ denote the 
maximal compact and \emph{connected} normal 
subgroup of $\fS$.  
For the following important fact, see \cite[Chapter 4, Theorem 1.7]{OVii} and \cite{Starkov}: 

\begin{proposition}[Bieberbach-Auslander-Wang Theorem] 
\label{prop:key-correct}
Let $\Gamma$ be a uniform
lattice in $G$. 
Then the intersection
$(\fR \, \fK) \cap \Gamma$ is a uniform lattice in 
$\fR  \, \fK$. In
particular, in the associated exact sequence
$$ 
 1\ra \fR\,  \fK \ra
 G \, 
 {\ra} \, S \ra 1 \; ,
 $$
the group $\Gamma$ projects to a uniform lattice in
a semisimple Lie group $S$ of noncompact type.
\end{proposition}

%
%

\subsubsection{Straightening of embeddings} \label{sect:straightening}
Let $\Gamma$ be a (uniform) lattice in a connected Lie group $G$ with Levi decomposition 
$$ G = \fR  \cdot  \fS   . $$ 
It is  a basic technique to modify a lattice $\Gamma$ in a controlled way to obtain another embedding of 
$\Gamma$ to $G$ which has possibly better properties.  
A useful result in this direction is the following 
Mostow deformation Theorem: 

\begin{proposition}[Deformation of lattices, \cf \mbox{\cite[Theorem 5.5 (b)]{Mostow3}}] \label{prop:Mostow_straight}
Let $G$ be a linear Lie group, and $\fD$ its maximal reductive normal connected subgroup, 
and\/ $\fF$ its maximal normal connected subgroup which does not contain 
any non-compact simple subgroups normal in $G$.
Then a finite index subgroup of\/  $\Gamma$ can be deformed into 
a uniform lattice $\Gamma'$ of $G$, where $$  \Gamma' = ( \Gamma' \cap \fF) \cdot  (\Gamma' \cap \fD  ) .$$ 
\end{proposition}

\subsection{Proper actions}
Let $G$ be a Lie group which acts 
on a locally compact Hausdorff space $X$.  
For any subset $A$ of $X$, we
put \[
\zeta_{A}(G)=\{g \in G\, |\, g \cdot A \cap
 A \neq\emptyset\} .
\]
The action of $G$ on $X$ is called \emph{proper} if for all compact 
subsets $\kappa \subset X$ the set $\zeta_{\kappa}(G)$ 
is compact. 

The following lemma is concerned with quotients of
proper actions. 

\begin{lemma}\label{lem:properquotient}
Let $G$ be a Lie group which acts properly on a smooth
manifold $X$. Let $R$ be a closed normal subgroup of $G$.
Then the quotient group $L= G/R$ acts properly on 
the quotient space $W=X/R$.
\end{lemma}

\begin{proof} Choose an arbitrary
compact subset $\bar \kappa$ from $W$ and 
consider its preimage $\mathcal K$ 
in $X$.  Since $R$ acts properly on $X$, there
exists a compact subset $\kappa$ of $X$ such that
$\mathcal K = R \cdot \kappa$. (Use the slice theorem
\cite{Palais,Koszul} for the action of $R$ on the manifold $X$.) 
It follows that $$ \zeta_{\mathcal K}(G) =  \zeta_{\kappa}(G) R .$$
Note that 
$\zeta_{\mathcal K}(G)$ projects onto  $\zeta_{\bar \kappa}(L)$
under the natural homomorphism $G \ra L$. 
Since $\zeta_{\bar \kappa}(L)$ is the image of the
compact set $\zeta_{\kappa}(G)$, $\zeta_{\bar \kappa}(L)$ is compact. 
\end{proof}

\subsubsection{Isometries of Riemannian manifolds} \label{oliverrework}
Let $X$ be a Riemannian manifold. 
By a theorem of Myers and Steenrod \cite{Myers-Steenrod}, the group \[
G = {\rm Isom}(X) 
\]
of isometries of $X$ acts properly on $X$.
For $x \in M$, let
\[
G_{x} = \{ h \in G \mid h x = x \}
\]
denote  the isotropy group at $x$. 
%
%
%
Let $\Gamma$ be a discrete
subgroup of ${\rm Isom}(X)$. 
\begin{lemma}\label{lem:isometry_lattices} Assume that 
$X/\Gamma$ is compact. Then the  group $\Gamma$ is a uniform lattice in ${\rm Isom}(X)$. In particular, 
$\Gamma \cap {\rm Isom}(X)^0$ is a uniform lattice in the connected component ${\rm Isom}(X)^0$. 
\end{lemma}
\begin{proof}
As ${\rm Isom}(X)$ acts properly, the quotient $X/{\rm Isom}(X)$ is
Hausdorff.  For the natural map $X/\Gamma\lra X/{\rm Isom}(X)$,
each fiber at a point $[x]\in X/{\rm Isom}(X)$ is homeomorphic to 
$\displaystyle {\rm Isom}(X)_x\backslash {\rm Isom}(X)/\Gamma$, and it is closed in $X/\Gamma$. As the
stabilizer ${\rm Isom}(X)_x$ is always compact and $X/\Gamma$ is
compact, it follows that ${\rm Isom}(X)/\Gamma$ is compact. 
In view of Lemma \ref{lem:lattice_in_component} this shows that $\Isom(X)^{0} \cap \Gamma$ is a uniform subgroup
of $\Isom(X)^{0}$. 
\end{proof}

\section{Smooth crystallographic actions} \label{sect:smooth_crystallographic}

Let $X$ be a contractible smooth manifold 
and $\Gamma$ a properly discontinuous group of diffeomorphisms of $X$. If the quotient space $$ X / \, \Gamma$$ is compact, we call the action of $\Gamma$ on $X$ crystallographic. In this section we are mostly concerned with the action of compact Lie groups on $X/\Gamma$.

\subsection{Cohomology of groups acting on acyclic spaces} 
Let $\Gamma$ be a group which has a
properly discontinuous cellular action  
on a CW-complex $X$. 
Let $R = \bbZ$ or $R= \bbZ_{p}$, 
for some prime $p$. The space $X$ is called acyclic over $R$ if
$H_i(X, R) = \{ 0 \}$, $i \neq 0$, and $H_0(X, R) = R$. 
(If  $R= \bbZ$ then $X$ is called acyclic.)
Let $R\,  \Gamma$ denote the group ring for $\Gamma$ 
with $R$-coefficients. We have the following important 
observation: 

\begin{proposition}  \label{prop:groupring_coho}
Assume that $X$ is acyclic over $R$ and $X/\Gamma$ is compact 
then $H^*(\Gamma, R \, \Gamma) = H^*_{c}(X,R)$. 
\end{proposition}

Here, $H^*(\Gamma, R \, \Gamma)$ denotes the group cohomology of $\Gamma$ with $R \,  \Gamma$ coefficients (see \cite{Brown} for definition)
and $H^*_{c}(X,R)$ denotes (cellular) cohomology of 
$X$ with compact supports. 
%
A proof of Proposition \ref{prop:groupring_coho}
for integral coefficients (that is, $R = \bbZ$) may be found in \cite[VIII, \S 7]{Brown} 
(see also  \cite[Lemma F.2.2]{Davis}).
The course of proof runs through almost 
verbatim using $R$-coefficients. 

\smallskip 
Every proper smooth action on a smooth manifold $X$ has an invariant
simplicial structure. In particular, if $\Gamma$ acts properly discontinuously and smoothly on $X$ then $X$ can be given the structure of a simplicial complex with simplicial action, \cf  
\cite[Theorem II]{Illman}.
We deduce:

\begin{theorem} \label{thm:groupring_coho}
Let $\Gamma$ be a group which acts smootly and properly discontinuously with 
compact quotient on an $n$-dimensional 
$R$-orientable smooth manifold $X$. If $X$ is acyclic over $R$ 
then $H^{n}(\Gamma, \Gamma R) = R$ and 
$H^i(\Gamma, \Gamma R) =
\{ 0 \}$, $i \neq n$. 
\end{theorem}
\begin{proof} By Poincar\'e duality for noncompact manifolds  
(cf.\ \cite[Theorem 26.6]{Greenberg}), 
$H^i_{c}(X, R) = H_{n-i}(X,R)$.
As $X$ is acyclic, we deduce
that $H^{n}_{c}(X, R)=R$ and $H^i_c(X,R) =\{0\}$, $i \neq n$. 
By Illman's result \cite{Illman}, we may assume that $X$ and the action of $\Gamma$
are cellular.  Therefore, Proposition \ref{prop:groupring_coho} implies the theorem. 
\end{proof}

\subsubsection{Smooth actions on manifolds}
Let $\Gamma$ be a group which acts properly discontinuously and smoothly with
compact quotient on the smooth manifold $X$. Under the assumption that $X$ is
acyclic (over the integers) we show that there are no compact group actions on $X$ 
which are normalized by $\Gamma$. Results of this type generalize the well known fact 
(see \cite{CR})  that the only compact Lie groups which act on compact aspherical 
manifolds are tori. 

\begin{lemma} \label{lem:torus_trivial}
Let $T$ be a compact $k$-torus  acting faithfully and smoothly on $X$, where $X$ is an orientable acyclic manifold. Assume that the action is normalized by $\Gamma$. 
Then $T =\{ 1 \}$. 
\end{lemma}
\begin{proof} 
By Theorem \ref{thm:groupring_coho},  we have that 
 $H^{n}(\Gamma, \Gamma \bbZ) =  \bbZ$, where $n = \dim X$, and $H^{i}(\Gamma, \Gamma \bbZ) =  \{ 0 \}$, $i \neq n$. 
Let $X^T$ be the fixed point set of $T$. 
Since $X$ is an acyclic manifold and the action of $T$ is smooth, Smith theory implies that the fixed point set $X^T$ is also acyclic. (See \cite[Chapter IV, Corollary 1.5]{Bredon}.) Note in particular that $X^T$ is non-empty and a connected manifold. By 
\cite[Chapter IV, Theorem 2.1]{Bredon} the manifold $X^T$ is also 
orientable. Since  $\Gamma$ normalizes $T$,  it acts on the acyclic manifold $X^T$ with compact quotient. Again, by  Theorem \ref{thm:groupring_coho}, we also 
have $H^{r}(\Gamma, \Gamma \bbZ) =  \bbZ$, $r= \dim X^T$.
This implies $\dim X = \dim X^{T}$, and, hence,  $T = \{ 1 \}$. \end{proof}

   
Closely related is the following fact (which is 
well known in the case that $\Gamma$ acts freely, 
compare e.g.\ \cite[Lemma 3.1.13]{LeeRaymond}): 
\begin{lemma}  \label{lem:finite_trivial}
Let $C \leq \Diff(X)$ be a finite group of diffeomorphisms of the smooth manifold $X$.
Assume further that $C$ is normalized by\/ $\Gamma$.  
If $X$ is acyclic and orientable over $\bbZ_{p}$,  where $p$ is prime, 
then $p$ does not  divide the order of $C$. In particular, if $X$ is acyclic and 
orientable then $C= \{ 1 \}$. 
\end{lemma}
\begin{proof}  Let $C({p})$ be a $p$-Sylow subgroup
of $C$. Since $X$ is acyclic over 
$\bbZ_{p}$, the Smith theorem (cf.\ \cite[Chapter III, Theorems 5.2, 7.11]{Bredon} implies 
that the fixed point set $X^{C(p)}$ is also acyclic over 
$\bbZ_{p}$. In particular, $X^{C(p)}$ is a connected manifold. It is also orientable over 
$\bbZ_{p}$,  by the proof of \cite[Chapter IV, Theorem 2.1]{Bredon}.
As $\Gamma$ normalizes $C$, it acts on the 
set of $p$-Sylow subgroups of $C$. Therefore, we may assume
(after going down to a subgroup of finite index in $\Gamma$) that
$\Gamma$ normalizes $C(p)$.
In particular, $\Gamma$ leaves $\displaystyle 
X^{C(p)}$ invariant. Since $X^{C(p)}$ is a connected manifold, which is acyclic and orientable over $\bbZ_{p}$, Theorem \ref{thm:groupring_coho} implies $\dim X = \dim X^{C(p)}$. Hence,  $C(p) = \{ 1 \}$
\end{proof} 
   
From the above, we deduce that properly discontinuous actions on acyclic \emph{smooth manifolds} do not normalize compact Lie groups: 

\begin{theorem} \label{thm:compact_trivial} 
Let $X$ be an orientable acyclic manifold and let $\Gamma$ be a group which acts smoothly and properly discontinuously on $X$ with compact quotient. Let $\kappa$ be a compact Lie group acting faithfully and smoothly on $X$,  such that the action is normalized by\/ $\Gamma$.
 Then $\kappa =\{ 1 \}$. 
\end{theorem}

\begin{proof}  
Since $\Gamma$ normalizes also the center of $\kappa$, which is
an extension of a toral group by a finite group, we may in the light of Lemma \ref{lem:torus_trivial} and Lemma \ref{lem:finite_trivial} 
assume from the beginning that $\kappa$ is connected and semisimple with trivial center. Moreover,  since $\kappa$ is semisimple, the group of outer automorphims $\Out(\kappa)$ is finite. Therefore, by going down to
a finite index subgroup of $\Gamma$ if necessary, we may
assume that conjugation with the elements of 
$\Gamma$ induces inner automorphisms of $\kappa$. 
As $\Gamma$ normalizes $\kappa$, the group
$\fD = \kappa \,  \Gamma$ is a Lie group of diffeomorphisms which acts
properly on $X$.
In view of the fact that 
the center of $\kappa$ is trivial, this implies that there exists 
a closed subgroup $ \Gamma' \leq \kappa \,  \Gamma $,  
which  \emph{centralizes} $ \kappa$,  such that 
$$  \fD = \kappa  \cdot  \Gamma = \kappa \cdot \Gamma' \, . $$
 (Indeed, given $\ell \in \fD$ define $\mu(\ell): \kappa \ra \kappa$ by
 $\mu(\ell) \, (k) = \ell k \ell^{-1}$, for all $k \in  \kappa$. 
For each $\gamma \in \Gamma$, let $k_{\gamma} \in \kappa$ be the unique element, such that $\mu({\gamma}) = \mu({k_{\gamma}})$.  Since the center of $\kappa$ is trivial, the map $\gamma \mapsto k_{\gamma}^{-1} \gamma$ is a homomorphism with kernel $\Gamma \cap  \kappa$.  Hence, the subgroup $$ \Gamma' = \{ \gamma' = k_{\gamma}^{-1} \gamma \mid \gamma \in  \Gamma \}$$ has the required properties.) 
Since $\Gamma'$ centralizes $\kappa$, $\kappa$ must be trivial, by
Lemma \ref{lem:torus_trivial}.
\end{proof}

We next turn to a related auxiliary result, 
which plays a role in Section \ref{sect:rad2}.

\begin{lemma} \label{lem:torus_trivial_connected_new} 
Let $\mathsf{p}: X \ra Y$ be a fiber bundle of contractible manifolds. Let 
$\Gamma$ act properly discontinuously on $X$ with compact quotient
and assume that the action descends equivariantly to a smooth action on $Y$.
Then any compact torus $T$ acting on $Y$ which is normalized by the image of $\Gamma$ 
acts trivially. 
\end{lemma} 
\begin{proof}
The fixed point set $Y^T$ is an acyclic orientable manifold (see Lemma \ref{lem:torus_trivial}).
Since the fibers of the bundle are contractible, the preimage $\mathsf{p}^{-1}(Y^T)$ is 
also acyclic. The latter is acted upon by $\Gamma$, since the action of $T$ is normalized by 
the image of $\Gamma$. Now, as in the proof of Lemma \ref{lem:torus_trivial}, 
$H^r(\Gamma, \Gamma \bbZ) = \bbZ$,  $r = \dim X$, and,  by Proposition \ref{prop:groupring_coho}, $H^i(\Gamma, \Gamma \bbZ) = \{0 \}$, 
$i >  \dim \mathsf{p}^{-1}(Y^T)$, so that $r = \dim \mathsf{p}^{-1}(Y^T)$.
Therefore, $Y= Y^T$, which
implies that $T$ acts trivially on $Y$. 
\end{proof}

\subsection{First application to Riemannian manifolds} 
It was proved in \cite[Claim II]{FW} that $ \Isom(X)^0$
contains no compact connected factor under the somewhat stronger assumption that $X/\Gamma$ is a manifold. By the above, this generalizes to divisible Riemannian manifolds $X$ and simultaneously strengthens the statement to include possibly non-connected compact normal groups: 
\begin{corollary} \label{cor:no_normal_compact}
Let $X$ be a contractible Riemannian manifold which 
is divisible. Then $ \Isom(X)$ has no non-trivial compact normal subgroup.   
 \end{corollary}

\section{Isometry groups with radical} \label{sect:Isomradical}
Let $X$ be contractible Riemannian manifold which is \emph{divisible},  that is,
there exists  a discrete group $\Gamma$ of isometries 
such that the quotient space $$ X \, \big/ \, \Gamma $$ is compact.
In this section, we are concerned with the
properties of the action of the continuous part  
$$ \fG= \Isom(X)^0 $$ 
of ${\rm Isom}(X)$ on $X$. 
In particular, we study the action of the maximal 
connected normal solvable subgroup
$$    \fR  \, \leq \, \fG \, , $$  
which is called the \emph{solvable radical} of $\fG $. 
One  main goal here is to show that the \emph{radical quotient} 
$$ X \big/ \, \fR $$ is a Riemannian manifold which is divisible by the image of $\Gamma$  in 
$\Isom(X/ \fR)$.
To this end, a detailed analysis for the action of $\Isom(X)$ on 
$ X/ \fR$ is required. 

\subsection{Principal bundle structure of the radical quotient} \label{sect:rad1} 
As our starting point, we show here  that the quotient $$ \mathsf{q} : X \lra  X/ \fR \;  $$
inherits a natural structure of a \emph{principal bundle}, where the base  is a contractible manifold and the structure group is a simply connected solvable Lie group $\fR_{0}$ which is contained in $\fR$. 

\smallskip 
Let $\mathsf N$ be the nilpotent radical 
of $\fR$. Then $\mathsf N$ 
is a closed and characteristic subgroup of  $\Isom(X)$. 

\begin{lemma}\label{lem:nilradical}
$\fN$ is simply connected.
 In particular,  
$\fN$ has no  (non-trivial) compact subgroup
and\/ $\fN$ acts properly and freely on $X$. 
\end{lemma}

\begin{proof}
Let $T$ be the maximal compact subgroup of $\fN$. Then $T$ is a compact torus which acts effectively on $X$. As $\fN$ is 
nilpotent, $T$ is central and 
characteristic in $\fN$, and also
characteristic in $\fG$. Therefore, 
$\Gamma$ normalizes $T$. 
Lemma \ref{lem:torus_trivial} implies that $T=\{1\}$.
\end{proof}
Therefore, $X/\fN$ is a contractible manifold  and $$X \to    X/\fN $$ is 
 a principal bundle with group $\fN$.

\smallskip   
We continue with our study of the action of $\fR$ on $X$. 
%
%
Let  $\fT$ be a maximal compact 
subgroup of $\fR$ (which is a compact torus).
By Lemma \ref{lem:nilradical}, $\fT$ intersects $\fN$ only trivially. Hence Lemma \ref{lem:chara-R0} asserts that we may \emph{choose} a simply connected solvable 
characteristic subgroup $\fR_0$ of 
$\fR$, with $\fN \leq \fR_{0}$, such that
\begin{equation}\label{1infra}
\fR=\fR_0\cdot  \fT \; \,  \text{and}  \; \,  \fR_0\cap  \fT = \{ 1 \}  .\end{equation} 
Since the
maximal compact subgroup of $\fR_0$ is trivial, 
$\fR_0$ acts freely on $X$. Indeed, the geometry of the $\fR$-orbits on $X$ is controlled 
by $\fR_{0}$, and, as the following Proposition shows, the quotient map 
$$\mathsf{q}: X \ra  X /\fR $$
is a principal bundle with group $\fR_0$. 

\begin{pro}\label{pro:solabel}
The compact torus $\fT$ acts trivially on  
$X/\fN$.  
Consequently, we have:  \begin{enumerate}
\item $\fR_{0}$ acts simply transitively on each fiber of $\mathsf{q}$.
\item $X /\fR$ is a contractible manifold.
\item $\fT$ is faithfully represented on $\fN$ by conjugation. 
\item For every $x \in X$, the stabilizer $\fR_{x}$ at 
$x$ is conjugate to $\fT$ by an element of\/ $\fN$.  
\end{enumerate}
\end{pro}

\begin{proof} 
The image of $\fT$
is the maximal compact subgroup of the abelian group 
$\fR/\fN$,
and it is a characteristic subgroup of $\fR/\fN$.
Since $\Isom(X)$ normalizes $\fR$,
it follows that the induced
action of $\Isom(X)$ on $$ Z = X/\fN $$ normalizes the 
induced action of $\fT$. 
We choose the unique Riemannian metric on $Z$, which
makes the map $X \ra Z$ a Riemannian submersion.
Since $\fN$ is characteristic in $\Isom(X)$, there is 
a  homomorphism  $$  \Isom(X) \lra \Isom(Z)  \; . $$ 
With respect to the induced actions, $\Gamma$ normalizes 
$\fT$. Therefore, 
Lemma \ref{lem:torus_trivial_connected_new}
implies that  $\fT$ acts trivially on $Z$. 
This implies that $ X/ \fR  =  X / (\fR_{0} \fT) =  X/ \fR_{0} \; .$
In particular,  (1) and (2) hold.  

%

Let $Z_{\fT}(\fN) \leq \fT$ denote the centralizer of $\fN$ in $\fT$.
Since $\fT$ acts trivially on the space $Z= X/\fN$, $\fT$ and therefore also 
its subgroup $Z_{\fT}(\fN)$)
act on each orbit $\fN \cdot {x}$, the latter centralizing the simply transitive 
action of $\fN $. Since $\fN$ is simply connected 
nilpotent,  $\fN$ does not contain any non-trivial compact 
subgroups. The same is true for the  centralizer of the simply transitive action of $\fN$. 
This implies that the compact group 
$Z_{\fT}(\fN)$  acts trivially
on each orbit of $\fN$. Taking into account that $Z_{\fT}(\fN)$ acts trivially on $X/\fN$, 
we deduce that $Z_{\fT}(\fN) = \{ 1 \}$, 
which proves (3).

As $\fR$ acts properly on $X$, the stabilizer $\fR_x$ is compact.
Since $\dim  \fR_x = \dim \fR - \dim \fR_{0} = \dim \fT$,  $\fR_x$ 
is  maximal compact in $\fR$, as well. By the conjugacy of maximal compact subgroups
of $\fR$, 
$\fR_x$ is conjugate to $\fT$, proving (4).
\end{proof}

%
%

\subsection{Divisibility of the radical quotient.} 
\label{sect:rad2}
Since ${\mathsf q}: X \ra Y$ is a principal bundle, we may equip the radical quotient $$ Y = X / \fR \,  $$ 
with the unique 
Riemannian structure such that $ {\mathsf q}: X \ra Y$  is a Riemannian submersion. 
Since $\fR$ is normal in $\Isom(X)$, we also have a well-defined 
homomorphism
\begin{equation}  \label{eq:homo_phi}
\phi: \,  \Isom(X) \lra \Isom(Y) \; . \\
\end{equation}
Let $$ \Theta = \phi(\Gamma)$$  denote the image of
$\Gamma$ in $\Isom(Y)$.  

\smallskip
The main goal of this subsection 
is to show:

\begin{theorem} \label{thm:Theta_dis}
$\Theta$ is a discrete subgroup of\/ $\Isom(Y)$.
In particular, $\Theta$ acts properly discontinuously on $Y$ and $Y/ \, \Theta$ is compact.  
\end{theorem} 
Hence, in particular, 
the Riemannian manifold $Y$ is divisible. 
\begin{remark} We may view Theorem \ref{thm:Theta_dis} as a parametrised version of 
Proposition $\ref{prop:key-correct}$.
\end{remark}

To prepare the proof of Theorem \ref{thm:Theta_dis}, 
we consider the kernel of the homomorphism $\phi$ in $\eqref{eq:homo_phi}$ (additional information
on $\ker \phi$ will be derived in Proposition \ref{prop:kerphi}
below). 

\smallskip
By construction of $Y$, the action of 
$\Isom(X)$ on $Y$ factors over the quotient $\Isom(X)/ \fR$.
Then we note: 

\begin{lemma} \label{lem:kerphi}
The image of\/ $\ker \phi$ in $\Isom(X) / \fR$ is compact.
In particular, $\ker \phi$ has finitely many connected components.
\end{lemma} 

\begin{proof} 
Note that, by Lemma \ref{lem:properquotient}, $\Isom(X)/ \fR$
acts properly on $Y$. In particular, 
the kernel of this  action is compact. This
shows that $\ker \phi$ has compact image in $\Isom(X)/ \fR$.
\end{proof}

Recall that, by 
Lemma \ref{lem:isometry_lattices}, $$ \Gamma_{0} = \fG \cap \Gamma$$ 
is a uniform lattice in $\fG = \Isom(X)^0$  and consider  the projection homomorphism
 $$\psi: \fG \ra \fG/ \fR \; . $$ 
Note that by the proof of Proposition  \ref{prop:key-correct}, the identity component 
$$ T = \ac{\psi(\Gamma_{0})}^{\, 0} \,  \leq  \;  \fG/ \fR$$  is a compact torus. 

\begin{lemma} $T$ acts trivially on $Y$.  \label{lem:TKtrivial}
\end{lemma} 
\begin{proof} 
By Proposition \ref{pro:solabel}, $Y$ is a contractible  manifold. 
The action of the compact torus  $T$ on $Y$ is normalized by the induced action of 
$\Gamma$ on $Y$. Hence,  Lemma \ref{lem:torus_trivial_connected_new} implies that $T$ acts trivially on $Y$.
\end{proof}


\begin{proposition}[Radical kernel is solvable lattice] \label{pro:Gam_kerphi}
The group $$ \Delta =  \Gamma \cap  \ker \phi$$ is a uniform
lattice in  $\ker \phi$. In particular, $\Delta$  
is a virtually polycyclic group.
\end{proposition}

\begin{proof} 
The image of  $\ker \phi$ is a normal 
subgroup of  $\Isom(X) / \fR$  and it is compact
by Lemma \ref{lem:kerphi}.  By Lemma \ref{lem:TKtrivial}, 
it also contains $$ T = \ac{\psi(\Gamma_{0})}^{\, 0} . $$  Hence,
by Lemma \ref{lem:uniformsubgs2}, $\Gamma \cap (\ker \phi)^0$
is a uniform lattice in $(\ker \phi)^0$. Since $\ker \phi$ has only
finitely many connected components, $\Gamma \cap \ker \phi$
is  a uniform lattice in  $\ker \phi$. Since $\ker \phi$ is a Lie group, 
which is an extension of a compact group by a solvable group, it follows
that the Lie group $\ker \phi$ is amenable.
As is well known, the discreteness of $\Gamma$ together with amenability of $\ker \phi$
implies that  $ \Gamma \cap  \ker \phi$ is virtually polycyclic.
(See, for example \cite[Lemma 2.2 (b)]{Milnor}.) 
\end{proof} 

We are ready for the 

\begin{proof}[Proof of Theorem \ref{thm:Theta_dis}] 
Since  $\Gamma \cap \ker \phi$ 
is a uniform lattice in $\ker \phi$,   the image $\Theta$ of
$\Gamma$ in $\phi(\Isom(X)) = \Isom(X)/ \ker \phi$ is a uniform lattice, see
Lemma \ref{lem:uniformsubgs}. 
Since $\Isom(X)/ \fR$ acts properly on $Y$, 
$\Isom(X)/ \ker \phi$ is a closed subgroup in
$\Isom(Y)$.  We conclude that $\Theta$ is discrete
in $\Isom(Y)$. Hence, $\Theta$ acts properly
discontinuously on $Y$, and, in particular, 
$Y /\Theta$ is Hausdorff. The natural surjective map $X/\Gamma \ra Y/\Theta$ shows that $Y/\Theta$ is compact.
\end{proof}

\subsection{Action of $\Isom(X)$ on the radical quotient} \label{sect:action_radicalquotient}
We choose a Levi decomposition (cf.\ Section \ref{sect:Levi_dec}) 
$$ \Isom(X)^{0} \, =  \, \fR \cdot \fS \, , $$ where $\fS$ is a semisimple subgroup and
 $\fR$ is the solvable radical of $\Isom(X)^{0}$. 
%
%
Moreover,  let 
 $$ \cK  \, \leq \, \fS $$ denote the maximal compact \emph{normal}
subgroup of $ \fS$. 

\smallskip 
Now  Theorem \ref{thm:Theta_dis}  combined with Theorem  \ref{thm:compact_trivial} implies:
\begin{proposition} \label{prop:Kisinthekernel}
The action of\/ $\cK$ on the quotient $Y$,  which is induced by $$ \phi: \Isom(X) \ra \Isom(Y)$$ is trivial. In particular, every finite subgroup in the center of\/ $\fS$  acts trivially on $Y$.
\end{proposition}

\noindent
That is, we have shown that $\cK$ is contained in $\ker \phi$. 

\smallskip 
Strengthening Lemma \ref{lem:kerphi}, 
we add the following observation: 

\begin{proposition} \label{ClaimA}
The image of\/ $\ker \phi$ in\/ $\Isom(X) / \fR$ 
is the unique maximal compact normal subgroup of\/ $\Isom(X) / \fR$.
In particular, $\ker \phi$ has finitely many connected components.
\end{proposition}
\begin{proof}  Let $\bar \cK_{1} $ be the image in $\Isom(Y)$  of any compact normal subgroup $\cK_{1}$ of $\Isom(X) / \fR$.  As before, let  $\Theta$ be the image of $\Gamma$.
By Theorem \ref{thm:Theta_dis}, $\Theta$ acts properly discontinuously with
compact quotient on $Y$ and $\Theta$ normalizes the compact group 
$\bar \cK_{1}$. 
Therefore Theorem  \ref{thm:compact_trivial}  implies 
$\bar \cK_{1} = \{ 1 \}$, which implies that $ \cK_{1} $ is contained 
in the image of $\ker \phi$. On the other hand, Lemma \ref{lem:kerphi} asserts that 
the image of $\ker \phi$ is compact. Since the image is also a normal subgroup of 
$\Isom(X) / \fR$, our claim follows. 
\end{proof}



%

Consider $\fR_{0} \leq \fR$ as in Proposition \ref{pro:solabel}.
\begin{lemma}\label{lem:difftrivial}
Let\/ $\fD$ be a compact Lie group of diffeomorphisms
of $X$ which centralizes $\fR$ and which acts trivially on $Y$.
Then $\fD = \{ 1 \} $. 
\end{lemma}

\begin{proof}  Since $\fD$ acts trivially on $Y$, $\fD$ 
centralizes the simply transitive action of $\fR_{0}$ 
on each fiber of the map $X \ra Y$. Since, $\fR_{0}$
is a simply connected solvable Lie group, its centralizer in the diffeomorphism group 
of the fiber (being isomorphic to $\fR_{0}$) does not contain any
non-trivial compact subgroup. 
 Hence the induced action of $\fD$ 
on each fiber must be trivial. Therefore, $\fD$ acts trivially
on $X$. This implies $\fD = \{ 1 \}$. 
\end{proof}

In the view of Proposition \ref{prop:Kisinthekernel}, we deduce:

\begin{proposition} \label{prop:nocompactfactor}
Any compact subgroup of\/ $\ker \phi$ acts faithfully on $\fR$ by conjugation.  In particular, this holds for the 
maximal compact normal subgroup  $\cK$ of\/ $\fS$.
\end{proposition}

We derive
some additional  observations on the kernel of 
$$  \phi: \Isom(X) \ra \Isom(Y) \; . $$  

\begin{proposition} \label{prop:kerphi}  
%
The  following hold:
\begin{enumerate} 
\item $\ker \phi$ acts faithfully on each fiber of\/ $\mathsf{q}: X \ra Y$. 
\item $(\ker \phi)^0 = \fR \cdot \cK^{0} \;  $. 
\item $\cK$ is contained in $\ker \phi $, and\/
$ \fR \cdot \cK$ is of finite index in $\ker \phi \cap \fG$. 
\end{enumerate} 
\end{proposition} 


\begin{proof} 
Recall that $\fR_{0}$ acts simply transitively on the fibers of 
$\mathsf{q}$ and let  $\fR_{0} \cdot p$, $p \in X$, be a fiber. Since 
$\fR_{0}$ is characteristic in $\Isom(X)$, $\ker \phi$ normalizes $\fR_{0}$. Consider the subgroup $T_{p}$ of elements in $\ker \phi$, which act trivially on the orbit $\fR_{0} \cdot p$. Since $\fR_{0} $ acts freely, $T_{p}$ centralizes $\fR_{0}$ in $\Isom(X)$.  Also $T_{p}$ is compact, since $\ker \phi$ acts properly on $X$. Therefore, with Proposition \ref{prop:nocompactfactor} we have (1).

Clearly,  $\fS \cap \ker \phi$ is a normal subgroup of $\fS$, and  by Proposition \ref{ClaimA}
, the homomorphism
$$   \ker \phi \; \to \; \Isom(X) /\fR $$ 
projects $\fS \cap \ker \phi$ 
onto the maximal compact normal subgroup $$ \bar \kappa \, \leq  \, \fG/\fR = \fS \big/ \fR \cap \fS \; .$$
The induced  surjective homomorphism $$ \fS \cap \ker \phi \to \bar \kappa$$  has kernel 
$ \fR \cap \fS$, which is central in $\fS$. 

Since $(\ker \phi)^0$ is a normal subgroup of $\fG = \Isom(X)^{0}$, $\fR \subseteq \ker \phi$ is also the solvable radical of $(\ker \phi)^0$.  Let $H$ be a Levi subgroup of $ (\ker \phi)^0$  with $H \leq \fS$. Since $$ (\ker \phi)^0 = \fR \cdot H ,$$ the Levi subgroup  $H = (\fS \cap \ker \phi)^{0}$ projects onto $\bar \kappa^{0}$, and 
is normal in $\fS$. Since the projection has discrete kernel, 
$\bar \kappa^0$ is semisimple (and also compact). Therefore, $\pi_{1}(\bar \kappa^{0})$ is finite. This shows that the covering group $H$ is  compact. 
Since  $\cK^{0}$ is the maximal compact, connected  normal subgroup 
in $\fS$, $H \subseteq  \cK^{0}$. On the other hand, by Proposition \ref{prop:Kisinthekernel},  $\cK \subseteq \ker \phi$. Therefore, $\cK^{0} \subseteq H$. This proves (2).

Since $\bar \kappa$ is compact, it has only finitely many connected components.
Since $\ker \phi \cap \fS$ projects surjectively on
the image $\bar \kappa$ of $\ker \phi$, using (2), it follows that  $\cK^{0} (\fS \cap \fR)$ is 
of finite index in $\ker \phi \cap \fS$. Hence, also $\cK (\fS \cap \fR)$
is of finite index in $\ker \phi \cap \fS$. Therefore,  (3) follows.
\end{proof} 

%
%
%
%
%

\smallskip
The following example shows that, in general, the maximal finite normal subgroup $C$ in the center of $\fS$ can be non-trivial, even if $\fS$ has no \emph{connected} compact normal subgroup.  Of course, this can only happen provided that the isometry group has a non-trivial connected radical on which $\fS$, respectively $C$, acts faithfully. (Indeed, by Proposition \ref{prop:Kisinthekernel}, the induced action of $C$ on the radical quotient $Y$ is always trivial.) However, even if the radical $\fR$ of $\Isom(X)$ is trivial, $\fS$ can have infinite center as in Example \ref{ex:sl2R}. 

\begin{example}[Compact locally homogeneous manifolds with radical] \label{ex:clhm}
\hspace{1ex} \newline 
Let $S$ be a semisimple Lie group  of non-compact type which
is faithfully represented on a vector space $V$. 
Since $S$ is linear, its center $C$ is finite. 
Let $K$ be a maximal compact subgroup of $S$.
Since $C$ is finite, $C$ is contained in $K$. Put 
$$   G =  V \rtimes S \; . $$
Now consider the contractible manifold 
$$    X \,  =  \,  V \times_{K} S  \,  =  G / K  \; . $$
Then $G$ acts properly and faithfully  on $X$. 
Choose a $G$-invariant
Riemannian metric such that $G = \Isom(X)^0$.
Since $V$ is the solvable radical, the radical quotient
$$ Y = X / V =  S/K$$ is a \emph{Riemannian symmetric space} 
of non-compact type. 
Now choose $V$ and a uniform subgroup 
$\Theta$ of $S$ such that $\Theta$ 
 is arithmetic with respect to a lattice $\Lambda$ in $V$.
(By a classical result of Borel \cite{Borel} such $\Theta$ and $V$ always exist.) 
Then the group $\Gamma = \Lambda \rtimes \Theta$ is a discrete uniform 
subgroup of $\Isom(X)$. This  constructs a corresponding compact locally homogeneous 
quotient $X/ \, \Gamma$ with $G = \Isom(X)^0$. 
\end{example}

\subsection{Divisible Riemannian manifolds}

Summarizing the above we arrive at the following structure 
theorem for the continuous part of the isometry group of a 
 contractible Riemannian manifold which is divisible. 

\begin{theorem}[Main theorem on radicals in  isometry groups]  \label{thm:structure}
Let $X$ be a contractible Riemannian
manifold and $\Gamma \leq \Isom(X)$ a discrete subgroup 
such that $X/\Gamma$ is compact.  Let $\fR$ denote the solvable radical of\/ $\Isom(X)^0$. 
Then the following hold:  
\begin{enumerate}
\item The maximal compact normal subgroup of\/ $\Isom(X)^0$  is trivial. 
\item The nilpotent radical $\fN$ of $\Isom(X)^0$ is simply connected and $\Gamma \cap \fN$ is a uniform lattice in $\fN$. 
\item The radical quotients  $ X/ \fN$, $X/\fR$ are contractible Riemannian manifolds.  
\end{enumerate}
With respect to the Riemannian quotient metric on $X/\fR$ and 
the induced homomorphism $$ \phi: \Isom(X) \ra \Isom(X/\fR)$$  it follows further  that:
\begin{enumerate}
\item[(4)]  The image $ \Theta$   of\/ $\Gamma$ in\/ $\Isom(X/\fR)$ is discrete. 
\item[(5)]  $\Isom(X)/\fR$ acts properly on $X/\fR$.  
\item[(6)] The kernel of the action (5) is  the maximal compact normal subgroup of\/ 
$\Isom(X) /\fR  $.
\item[(7)] The image $S$ 
of\/  $\Isom(X)^0$ in $\Isom(X/\fR)$ is a 
semisimple Lie group $S$ of non-compact type without finite subgroups in its center, and 
it is a closed normal subgroup of\/ $ \Isom(X/\fR)^{0}$. (In particular, it is normal in a
finite index subgroup of\/ $\Isom(X/\fR)$.) 
\item[(8)]
Moreover, $\Theta \cap S$ is a uniform lattice in $S$. 
\end{enumerate}
\end{theorem}
\begin{proof}  Recall that (1) is a consequence of Corollary \ref{cor:no_normal_compact},  and (3) is proved in Proposition \ref{pro:solabel}. Next (4) and (5) are contained in Theorem  \ref{thm:Theta_dis} and its proof, whereas (6) is Proposition \ref{ClaimA}. 
For (2), observe that by (1), $\Isom(X)^{0}$ has no connected compact normal semisimple subgroup.  Since $\Gamma \cap \Isom(X)^{0}$ is a lattice in $\Isom(X)^{0}$, a result of Mostow \cite[Lemma 3.9]{Mostow3} shows that $\Gamma \cap \fN$ is a uniform lattice in $\fN$.

To prove (7), we first observe that $S= \phi(\Isom(X)^0) $  acts properly on $X/\fR$, by (4). In particular, $S$ must be a closed subgroup of $\Isom(X/\fR)$. By Proposition \ref{prop:kerphi}, the maximal compact normal subgroup
 of $\fS$, where $\fS$ is a Levi-subgroup of $\fG=\Isom(X)^0$,  is contained in $\ker \phi$. 
 Therefore, $S= \phi(\fS)$, is semisimple of non-compact type. By (3), $X/\fR$ is divisible by $\Theta = \phi(\Gamma)$, and $S$  is normalised by this action.  By Theorem \ref{thm:compact_trivial}, $S$ must be without finite normal subgroups. 

By the Malcev Harish-Chandra theorem \cite[Ch.\ III, p.92]{Jacobson}, all Levi subgroups of a Lie group are conjugate by elements of its nilpotent radical. Therefore, to show that $S$ is normal in $\Isom(X/\fR)^{0}$, it is enough to show that $S$ is centralised by the nilpotent radical $\fN_{{X/\fR}}$ of $\Isom(X/\fR)^{0}$. Let $\Theta_{1} =  \fN_{{X/\fR}} \cap \Theta$,  where $\Theta = \phi(\Gamma)$. 
Since $\Theta$ normalises $S$, $[S, \Theta_{1}] \subseteq S \cap \fN_{{X/\fR}}  = \{ 1 \}$. Moreover,
 since $\Theta$ divides $X/\fR$, $\Theta \cap \Isom(X/\fR)^{0}$ is a uniform lattice in $\Isom(X/\fR)^{0}$. By (2), $\Theta_{1}$ is thus a uniform lattice in the simply connected nilpotent group $\fN_{{X/\fR}}$. Since $S$ centralises the uniform lattice $\Theta_{1}$, 
it also centralises $\fN_{{X/\fR}}$. Hence, $S$ is normal in $\Isom(X/\fR)$.

Since the quotient of $X/\fR$ by $\Theta$ is compact, $\Theta$ is a uniform lattice in $\phi(\Isom(X))$, which 
acts properly on $X/\fR$, by (5). By Lemma \ref{lem:lattice_in_component}, $\Theta$ intersects $S$ as a uniform lattice. So (7) is proved. 
\end{proof}

\section{Infrasolv towers}\label{sec:soltower}

\subsection{Review of infrasolv spaces} \hspace{1ex}

\smallskip
\paragraph{\em Infra $R$-geometry} 
Let $R$ be a connected Lie group
and $\Aff(R)$ its group of affine 
transformations. By definition, $\Aff(R)$ is precisely 
the normalizer of the left translation action of
 $R$ on itself in the group of all diffeomorphisms
 of $R$. The group  $\Aut(R)$ of 
continuous automorphisms of $R$ is then naturally a subgroup of 
$\Aff(R)$. Identifying $R$ with its group of \emph{left} 
translations gives rise to a semi-direct product decomposition 
$$ \Aff(R) = R \rtimes \Aut(R) \; . $$ The associated projection homomorphism 
$$ \Aff(R) \to \Aut(R) \; , $$ 
is called \emph{holonomy homomorphism}. 
Let $$ \Delta \leq \Aff(R)$$  be a \emph{discrete} 
subgroup which acts properly discontinuously on $R$.
Then the quotient space $$R/\Delta$$ is called an 
\emph{orbifold with $R$-geometry}.

\smallskip 
\paragraph{\em Infrasolv spaces}
We assume now that $R$ is a simply connected solvable
Lie group. Traditionally, a space with $R$-geometry is called 
infrasolv if it admits
an underlying compatible Riemannian geometry. 
 
\begin{definition} \label{def:infrasolv}
A \emph{compact}  space $R/\Delta$ with $R$-geometry is  
called an \emph{infrasolv orbifold} if the closure of 
the holonomy image of $\Delta$ in $\Aut(R)$ is compact. 
If the infrasolv orbifold $R/\Delta$ is a manifold, 
it is called an \emph{infrasolv} manifold. 
\end{definition}

\noindent Equivalently, we can say that 
\emph{$R/\Delta$ is infrasolv, if and only if the closure of $R \Delta$ in 
$\Aff(R)$ acts properly on $R$.} In particular, if $R/\Delta$ is infrasolv, there exists a  $\Delta$-invariant left-invariant Riemannian metrics on $R$, and any such metric gives rise to an associated \emph{infrasolv Riemannian structure} on $R/\Delta$. 

\begin{lemma} \label{lem:infrasolv} 
Any  Riemannian space 
$X$ with a simply transitive isometric action of  $R$, and a (cocompact) properly discontinuous subgroup $\Delta$ of isometries of $X$, which is normalising this action, constructs an \emph{infrasolv orbifold with $R$-geometry}. This infrasolv structure is unique up to a right-multiplication of $R$. 
\end{lemma}
\begin{proof} Let $\Aff(X,R)$ denote the normaliser of $R$ in $\Diff(X)$.  
Fixing a point $p \in X$ defines identifications $$ R = X \; \text{ and }  \;  \Aff(X,R) =  R \rtimes \Aut(R)  .$$
 Here,  $\Aut(R)$ corresponds to the 
stabiliser of $p$. For $\delta \in \Delta$, let 
$$ \delta = r \phi \; , \, r \in R, \,  \phi \in \Aut(R) $$ be its corresponding 
holonomy decomposition. Observe that $\phi$ is an isometry, since $\delta$, $r$ are.
Recall further that the linear isotropy representation of $\Aut(R)$ 
on the tangent space $T_{p} X$ at $p$ is an isomorphism of $\Aut(R)$ onto a closed
subgroup of $\GL(T_{p} X)$. Since $\phi$ is an isometry, its linear isotropy image 
is contained in a compact subgroup of $\GL(T_{p} X)$. This shows that the holonomy
image of $\Delta$ has compact closure in $\Aut(R)$. 
\end{proof}


\subsection{Infrasolv fiber spaces} \label{sect:fiber_spaces}
We introduce now a notion of orbibundles whose fibers 
carry an affine geometry modeled on the solvable Lie 
group $R$. We use a characterisation in terms of group 
actions on the universal cover.
\smallskip

Let $X$ be a simply connected manifold on which $R$
acts properly and freely. Then let $\Diff(X,R)$ denote the
normalizer of $R$ in $\Diff(X)$.  Consider a subgroup
 $$ \Gamma \,  \leq \, \Diff(X,R)  $$  
which acts properly discontinuously on $X$.
Let $\Delta \leq \Gamma$ be the normal subgroup of $\Gamma$, which
acts trivially on the quotient
manifold
$$ Y= X/R \; . $$ Put $$ \Theta = \Gamma/\Delta$$ and assume
further that \emph{$\Theta$ acts properly discontinuously on $Y$.} 
Thus, in this case,  $Y/\Theta$ is a Hausdorff space, in fact, it is an orbifold.  

\smallskip 
We consider the  projection map
$ \mathsf{p}: X/\Gamma \ra Y/\Theta$, 
induced by $ \mathsf{q}: X \to Y$. 
Since $\Theta$ acts properly discontinuously, the fiber stabilisers 
$$ \Delta_{y} = \{ \gamma \in \Gamma \mid \gamma y =y \} , \; y \in Y, $$ 
are finite extension groups of $\Delta$. Since $\Gamma \leq \Diff(X,R)$ normalizes $R$, the restriction of $\Delta_{y}$ to 
$ \mathsf{q}^{-1}(y)$ acts by affine transformations with respect to the simply transitive action of $R$ on $ \mathsf{q}^{-1}(y)$. 
Therefore, the fibers $$ \mathsf{p}^{-1}(\bar y) = R/ \Delta_{y}$$
carry the structure of a space with $R$-geometry. 

\begin{definition}\label{fibersolv} The projection map $$ \mathsf{p}: X/\Gamma \ra Y/\Theta$$  is called a \emph{fiber bundle with $R$-geometry} over the base $Y/\Theta$.  It is  
called an \emph{infrasolv fiber space} if the fibers of 
$\mathsf{p}$ are (compact) infrasolv orbifolds. 
\end{definition}

Given an infrasolv fiber space $\mathsf{p}$, and,  in addition,   
Riemannian metrics on $X$ and $Y$, invariant by $R$, $\Gamma$ and $\Theta$, 
respectively, such that  $$  X \to Y$$  
is a Riemannian submersion, then 
we call 
$\mathsf{p}$ a \emph{Riemannian 
infrasolv bundle} (or fiber space) modelled on the  group $R$. 


\subsection{Structure theorems} 
Applying the results in Sections  \ref{sect:rad1} and 
\ref{sect:rad2},  we can state now: 

\begin{theorem} \label{thm:basic_bundle}
Let $X$ be a contractible Riemannian manifold which is divisible, and let\/ 
$\Gamma \leq \Isom(X)$ be  a 
discrete  subgroup such that $X/\Gamma$ is compact.
Let $\fR$ be the solvable radical of\/ $\Isom(X)$ and 
put $Y = X / \fR$. Let $\Theta$ denote the homomorphic image of\/  
$\Gamma$ in\/ $\Isom(Y)$. Then $X/\Gamma$ has an induced structure of 
Riemannian infrasolv fiber space over the compact aspherical
Riemannian orbifold $Y/\Theta$.
\end{theorem} 
\begin{proof} 
By (1) of Proposition \ref{pro:solabel}, there exists
a simply connected characteristic subgroup $\fR_{0}$ of $\fR$, 
acting properly and freely on $X$ such that $Y = X/ \fR_{0}$ is a contractible manifold. Moreover, the metric on $X$ descends to $Y$ such that the map 
$X \ra Y$ is a Riemannian submersion.  Since $\fR_{0}$ is
characteristic in $\Isom(X)$, we have $\Gamma \leq \Diff(X,\fR_{0})$. 
Consider the associated homomorphism $\phi: \Isom(X) \ra \Isom(Y)$. Then the image $\Theta =  \phi(\Gamma)$ acts properly discontinuously,  by Theorem \ref{thm:Theta_dis}. Hence, the map  $\mathsf{p}: X/\Gamma \lra Y/\Theta$ is a fiber bundle with $\fR_{0}$-geometry over the base orbifold $Y/\Theta$. 

Obviously, since $X/\Gamma$ is compact, 
the fibers of $\mathsf{p}$ are compact.
It remains to show that the bundle is \emph{infrasolv}. That is, we have to
show that the holonomy of every fiber has compact closure. 

For this, note that the metric of $X$ restricts to a $R\Delta_{y}$-invariant 
Riemannian metric on the $R$-orbit $\mathsf{q}^{-1}(y)$. 
As remarked in Lemma \ref{lem:infrasolv}, 
this implies that the fibers are infrasolv. 
\end{proof}

\begin{remark}[Infrasolv structure on the fibers]  The infrasolv structure on the fibers depends on the choice of group $\fR_{0}$, which is not necessarily unique. 
\end{remark} 

\begin{remark}[Addition to \mbox{Theorem \ref{thm:basic_bundle}}] 
The geometry of the fibers of the bundle $\mathsf{p}$ constructed in the proof of Theorem \ref{thm:basic_bundle}, is determined by the holonomy image of $\Delta$ in the fibers of the map
$X \to Y$. Since $\Delta \leq \ker \phi$, Proposition \ref{prop:kerphi} part (3), shows that, \emph{up to finite index},
the holonomy of the fibers is contained in the holonomy image of the subgroup $\fR \cdot \cK \leq \Isom(X)$.  
\end{remark}

The theorem applies in particular to compact aspherical Riemannian \emph{manifolds}
 $$ M = X / \Gamma\, . $$
Since $M$ is a manifold, $\Gamma$ acts freely on $X$, and
$\Gamma \leq \Isom(X)$
is a discrete torsion-free subgroup isomorphic to the fundamental group $\pi_{1}(M)$. Then  
$M$ inherits the structure of an infrasolv fiber space 
$$   M  \lra Y / \Theta $$ 
over the base $Y / \Theta$ which, in general, is an  \emph{orbifold}.  However, the fibers of the bundle map are always infrasolv manifolds, since 
$\Gamma$ acts freely on $X$.  
 

\subsection{Towers of infrasolv fiber spaces} 
Using Theorem \ref{thm:basic_bundle} we construct a sequence of Riemannian submersions 
$$  \mathsf{q}_{i}:  X \lra X_{i} \; , $$ 
subsequently dividing by the solvable radical $\fR_{i}$ of $\Isom(X_i)$.  That is, assuming that $\mathsf{q}_{i}$ is constructed,  we put $$ X_{i+1} = X_{i} /  \, \fR_{i} \;   $$
and define $\mathsf{q}_{i+1}:  X \ra X_{i+1}$ as the 
composition of $\mathsf{q}_{i}$ with 
the projection $X_{i} \ra X_{i+1}$. 
In this way, we obtain a tower of Riemannian submersions 
\begin{equation}  \label{eq:tower1} X \lra  X_1 \lra \,  \ldots \, \lra X_{k} \; ,  \end{equation}
and the induced maps 
$ \Isom(X_{j}) \ra \Isom(X_{j+1})$,  $ j+1 \leq  i$, 
compose to homomorphisms
$$\phi_{i}: \Isom(X)  \ra  \Isom(X_{i}) \; ,   $$
such that $$ \Gamma_{i}  \, =  \, \phi_{i}(\Gamma) $$
acts properly discontinuously and with compact quotient 
$$   X_{i} / \, \Gamma_{i} \; . $$ 
If, for some $\ell$,  
$\Isom(X_{\ell})$ has trivial solvable radical, the process
terminates, and we call $\ell$ the \emph{length} of the tower.   

\smallskip 
In view of  Theorem \ref{thm:structure} and Theorem \ref{thm:basic_bundle},  the following properties are satisfied: 
\begin{corollary}[Infrasolv tower for $X$] \label{cor:solvtower} \hspace{1ex} 
\begin{enumerate}
\item $X_{i}$ is a contractible Riemannian manifold
and the projection $X_{i} \ra X_{i+1}$ is  a principal bundle with structure group 
a simply connected solvable Lie group.
\item The maps $X_{i} /\, \Gamma_{i} \ra X_{i+1}/ \, \Gamma_{i+1}$, $i+1 \leq \ell$, 
are Riemannian  infrasolv fiber spaces.
\item  $\Isom(X_{\ell})^0$ is a semisimple (or  trivial) Lie group of non-compact type, which has no non-trivial finite subgroups in its center. 
\end{enumerate}
\end{corollary}
\noindent 
Any sequence of maps 
\begin{equation}  \label{eq:tower2} 
X/\Gamma \lra  X_1/ \Gamma_1 \lra \,  \ldots \, \lra X_{k}/\Gamma_k  \; ,  \end{equation}
such that at each step $X_i / \Gamma_i \to X_{i+1}/\Gamma_{i+1}$ is a Riemannian infrasolv bundle, will be called an \emph{infrasolv tower} for $X$ over the  \emph{base} 
$$ X_{k}/\Gamma_k  \; . $$ 

\medskip
We call the tower \emph{complete} if the solvable radical of $ \Isom(X_{k})^{0}$ is trivial. 
Thus for a complete tower, $$ \fS =  \Isom(X_{k})^0$$  is semisimple and the center of $\fS$ is a finitely generated abelian group (which is torsion-free, by Corollary \ref{cor:no_normal_compact}.)  
Let $r_k$ denote the rank of $\cZ$, where $\cZ$ is the center of $\fS$. 
Set $X=X_0$. 

\begin{definition} \label{def:solvrank}
The \emph{solvable rank} of the  complete  infrasolv tower 
is the integer 
\[
r=\mathop{\sum}_{i=0}^{k-1}({\rm dim}\, X_{i}-{\rm dim}\,
X_{i+1}) +  r_k ={\rm dim}\, X-{\rm dim}\, X_{k} + r_k   \; . \] 
\end{definition}
The definition is motivated by Example \ref{ex:sl2R}.

\begin{corollary} \label{cor:solvrank}
The group $\Gamma$ contains a 
normal polycyclic subgroup of rank equal to $r$.
\end{corollary}
\begin{proof}  Indeed, we have that  $\rank \Gamma_{i}/\Gamma_{i+1} = \dim X_{i} - \dim X_{i+1}$.
(Recall that any virtually polycyclic group which acts properly discontinuously with
compact quotient on a contractible manifold $X$ has virtual cohomological dimension 
$\mathrm{vcd} \,  \Gamma = \dim X$, \cf Theorem \ref{thm:groupring_coho}. Furthermore,
$\rank \Gamma = \mathrm{vcd}\,  \Gamma$, see \cite{Brown}.)
\end{proof}

%
%
%


\section{Aspherical manifolds with large symmetry} \label{sect:examples} 
We give some examples of Riemannian metrics on aspherical manifolds which exhibit 
various types of local symmetry.  

\subsection{Warped product metrics}
A special case of Riemannian submersions $X \ra Y$ 
are warped products $X = Y \times_{f} F$ of Riemannian 
manifolds $Y$ and $F$ , where $f: Y \ra \bbR^{>0}$ denotes the
warping function. 
The manifolds  $(y \times F)$, $y \in Y$, are 
called the fibers of the warped product. The following lemma describes fiber preserving warped product isometries: 

\begin{lemma} \label{lem:wp_isometries}
Assume that the function $f$ is bounded. Then 
every isometry\/ $\Phi$ of $X$ which 
maps fibers
to fibers, is of the form 
$\Phi = \psi \times \phi$, where $\phi \in \Isom(F)$, 
and $\psi \in \Isom(Y)$ satisfies $f \circ \psi =  f$.
\end{lemma}
\begin{proof} 
Indeed, since $\Phi$ is an isometry which preserves the fibers, it induces an isometry $\psi$ of the base $Y$. 
It also respects the horizontal distribution of the warped product, which is tangent to the horizontal leaves $Y \times x$, $x \in F$. Therefore, $\Phi = \psi \times \phi$, for some map $\phi: F \ra F$.
Writing  $g = g_{Y} \times f \, g_{F}$ for the warped product metric, we see that $\Phi$ is an isometry if and only if, for all $y \in Y$, $v \in T_{x} F$,  $$ f( \psi(y))  \,  g_{F,\phi (x)} (d \phi_{x}(v), d \phi_{x}(v)) = f(y) \, g_{F,x}(v,v) \; . $$   
Therefore, keeping $x$ fixed, we deduce that there exists a unique $\lambda = \lambda(\psi) > 0$ such that the relation $$  f( \psi(y)) \, = \,  \lambda \, f(y)$$ holds, for all $y \in Y$.  Consequently,  $\phi$ is a homothety for $g_{F}$ with factor $\lambda^{-1}$. Clearly, the map $k \mapsto \lambda(\psi^{k})$ defines a homomorphism from a cyclic group to $\bbR^{>0}$. Hence, if $f$ is bounded it  follows that $\lambda=1$. 
\end{proof}

Compact aspherical manifolds of large symmetry are easily constructed using warped product metrics. The following simple example exhibits a compact aspherical Riemannian surface $M_{f}$, diffeomorphic to the two-torus, whose metric has large symmetry, in the sense of Definition \ref{def:symmetry}. Its associated solvtower is of length two. (In particular, the metric is not locally homogeneous.)

\begin{example}[Torus of revolution]\label{ex:revolution_torus}
Consider the warped product of circles  
$M_{f}= S^1 {\times}_{\bar f} \,  S^1$ which is covered
by $ X= \RR \times_f^{} \RR$, where 
\begin{equation} \label{eq:circle warp}
 f(x) = 2+ \sin x . \end{equation} 
Accordingly, $X$ has metric $g=  dx^2 + f dy^2$. Then 
$\Isom(X)$ contains the translation group $V= \bbR \times 2\pi\ZZ$, and putting 
$\Lambda  =\ZZ \times 2\pi\ZZ$, we have 
that $$ M_{f}= X/\Lambda. $$ We claim that $V^{0} = \RR$ must be the nilpotent radical of 
$\Isom(X)^{0}$. 
Note that $\Isom(X)^{0}$ has no semisimple part (otherwise $X$ is isometric to the hyperbolic plane, which is absurd since $T^{2}$ has no metric of negative curvature.)  Therefore, $\Isom(X)^{0}$ is solvable. We can infer from Lemma \ref{lem:wp_isometries} that the nilpotent radical is one-dimensional. Indeed, since the nilradical acts freely, it is at most two-dimensional. In particular, it must be abelian, and therefore respects the fibers of the warped product. 
Therefore, the radical quotient for $X$ coincides with the warped product projection   
$$  X \ra  (\RR,dx^2). $$ It gives a Riemannian submersion over the real line,
whose fibers are the orbits of $\Isom(X)^{0} =V^{0}$, which 
are permuted by $\Isom(X)$.
\end{example}

More generally, for \emph{any\/} Riemannian manifold $N$,  we can form the warped product $$ M_{N,f} =  S^1 {\times}_{\bar f} \,   N \; , $$ \emph{where $\bar f$ is as in \eqref{eq:circle warp}}.
Take, for example, a (compact) locally symmetric space of noncompact type $$ N =  \Theta \backslash S/K  $$ 
as fiber. Put
$$     X =  B  {\times}_{ f} \,   S/K $$ 
for  the lifted  warped product on the universal cover, where $B =\RR$ is the real line.
By Theorem \ref{thm:basic_bundle} and Proposition \ref{prop:kerphi},  $S$ acts locally faithfully on the radical quotient of $X$. It follows that the nilpotent radical of $\Isom(X)^{0}$ is at most one-dimensional, and it is therefore centralized by $S$. In particular, the nilradical of $\Isom(X)^{0}$ acts on the fibers of the warped product, which are the orbits of $S$. We deduce from Lemma \ref{lem:wp_isometries} that the nilradical of  $\Isom(X)^{0}$ must be trivial. Hence, $$ \Isom(X)^{0} = S$$ is semisimple, and the base of the lifted  warped product 
$$     X =  B  {\times}_{ f} \,   S/K $$ 
is a  Euclidean space.  We proved: 

\begin{pro}\label{prop:wp2}
The  compact aspherical Riemannian warped product 
$$ M_{\Theta \backslash S/K,\bar f} \to  S^{1}$$ with $f$ as in \eqref{eq:circle warp} satisfies $\Isom(X)^{0} = S$ is semisimple. Moreover,  the isometry group of the base of the warped product metric on $X$ is $$ \Isom(B) = \Isom(\bbR) \; . $$
\end{pro}

\subsection{Metrics on\/ $\tSL$ and its compact quotients} \label{sect:exsl2R}

Let $$ \tSL$$  denote the universal covering group of $\SL(2,\RR)$. 

\begin{example} \label{ex:sl2R} 
Consider  the group manifold $X = \tSL$ with any left-invariant Riemannian metric. Then by the left action on itself 
 $$\tSL \leq  \Isom(X)$$ is represented as a subgroup of the isometry group.  
 Let $\fR$ denote the solvable radical of $\Isom(X)^0$ and $\cZ$ the center of  $\tSL$. 
 Two principal cases do  occur:
\begin{enumerate}
\item[(I)] 
The radical  $\fR$ is a one-dimensional vector group and acts freely on $X$. The  Riemannian quotient 
 $ X/ \fR$ is (up to scaling) isometric to the hyperbolic plane  $ \mathbb{H}^2$, so that the associated fibering 
  $$  X_{\tSL} \, \to \;    \mathbb{H}^2 $$ is a Riemannian submersion,  and 
$$ \Isom(X)^0 =  \fR \;   \tSL \; ,  \; \fR \cap  \tSL   =  \cZ .$$ 
 
\item[(II)] 
$\Isom(X)^0 = \tSL$. 
\end{enumerate}

\begin{proof}  
Observe that $\Isom(X)^0$ is reductive with Levi subgroup  $\tSL$.
Assume that  the radical $\fR$ is non-trivial. Since $\fR$ is abelian and centralising $\tSL$, 
it arises as a one-parameter group of right-translations on $\tSL$. Since $\Isom(X)^0$ acts properly, the adjoint representation of this one-parameter group defines a compact subgroup of inner automorphisms of the Lie algebra of $\SLt$. It also preserves the scalar product on the
tangent space of the identity, which defines the left-invariant Riemannian metric on $X$. 
In particular, $\fR$ arises  from the subgroup $\tSO$ in  $\tSL$, which is covering (a conjugate of)  $\SO(2)$. 
Moreover, $X/ \fR = \SLt/ \SO(2) = \HH^{2}$. This is type (I). 
%
\end{proof}
\end{example}
\begin{remark} \label{rem:typeIIsl2}
The second case actually occurs and is the generic one.  For example,  take a standard basis $X,Y, H$ of the Lie algebra of\/  $\SLt$ with $[H, X ] =  2 X$, $[H, Y ] =  -2 Y$ and  $[X,Y ] = H$.
Define it as orthonormal basis for the scalar product at the tangent space at the identity. 
We verify 
that this scalar product is not preserved by any compact subgroup in the adjoint image of\/ $\SLt$. Hence, this defines a left-invariant metric on $\tSL$, which is of type (II). 

\begin{proof}[Proof of remark] We show that the scalar product does not admit a one parameter group of inner isometries. 
Let $B = a H + bX + c Y  \in \lie{sl}(2,\bbR) $, where $a,b,c$ are real numbers. Then $\mathrm{ad}(B): \lie{sl}(2,\bbR) \to \lie{sl}(2,\bbR)$ is represented, with respect to the basis $\{ H,X,Y\}$, by a matrix of the form $$
\left( \begin{matrix}
0 & -c & b \\ -2b & 2a & 0 \\ 2 c & 0 & -2a	
\end{matrix} \right) \; . 
$$  Now this matrix is skew if and only if $B=0$.
\end{proof}
\end{remark} 

\paragraph{\em Associated principal circle bundle over K\"ahler manifolds} 
Now let $\Gamma \leq \tSL$ be a (uniform) lattice.  Then $\Delta= \Gamma \cap \cZ$ has finite index in $\cZ$ (by \cite[5.17 Corollary]{Raghunathan}) and $\Gamma$ projects to  a uniform lattice $$ \Theta = \Gamma/\Delta \leq \PSL_2\bbR  = \Isom(\mathbb{H}^2)^0 .$$ 
Let  $\Gamma$ act by left-multiplication on $\tSL$ and put 
$$ M = X / \, \Gamma = \Gamma \, \backslash \tSL \; . $$ 
Therefore: 
\begin{enumerate}
\item[(I)] (the Sasakian case):  Since $\Isom(X)^0$ has a radical $\fR$ isomorphic to the real line, there is an induced 
infrasolv-tower of length $k= 1$ over a compact hyperbolic orbifold:
$$   M  \lra  \,  (X/ \, \fR = \mathbb{H}^2) / \, \Theta . $$
The map is actually a principal circle bundle, and a Riemannian submersion. These manifolds $M$ and their geometry play a prominent role in the classification of three manifolds. Compare \cite{Geiges}.
\item[(II)]
$\Isom(X)^0 = \fS= \tSL$ is
semisimple with infinite cyclic center $\cZ$. The Riemannian manifold $X$ 
therefore does not admit an infrasolv-fibering. 
\end{enumerate}

In both cases,  $M$ is locally homogeneous (in particular $M$ has large symmetry) and the solvable rank  
of the metrics  (see Definition \ref{def:solvrank}) satisfies $r =1$.

\section{Constructing Riemannian manifolds from group extensions}\label{sect:group_extensions}
In this section we  introduce a method which allows to construct examples of aspherical 
Riemannian orbibundles 
$$ \mathsf{p}:   X /\Gamma \to Y/ \Theta  \; , $$ 
which arise from associated  group extensions of the form
$$ 1 \to \Lambda \to  \Gamma \lra \Theta \ra 1 \, . $$ 
In our setup,  $\Lambda$, in general,  will be a virtually polycyclic group. We  will 
work out the details only in a specific, particularly simple case. Our main purpose here 
is to construct an aspherical manifold, which supports Riemannian metrics  
of large symmetry, but at the same time does not admit any 
locally homogeneous Riemannian metric, see Corollary \ref{cor:non_locally_homogeneous}. 
The method though generalises 
considerably. The basic construction is partially based on the notion 
of \emph{injective Seifert fiber spaces} (\cf \cite{LeeRaymond} for an extensive account). 

\subsection{Fiberings over hyperbolic Riemannian orbifolds}
\hspace{0.1cm} 

\smallskip 
\paragraph{\em Setup}
As a  base space of our tentative fibration, we choose a hyperbolic orbifold 
(that is, a space of constant negative curvature) $$ \HH^n / \, \Theta \;  . $$
Here, $\Theta \leq {\rm PSO}(n,1)$ is a discrete uniform subgroup, and  
$$ {\rm PSO}(n,1) = {\rm
Isom}(\HH^n)^0$$  denotes the identity component
of the group of isometries of real
hyperbolic space $\HH^n$.  
In addition, we are considering a central group extension:
\begin{equation}\label{centralex}
1\ra \ZZ^k\ra \Gamma \lra \Theta \ra 1 \; \, .
\end{equation} 
Up to isomorphism \eqref{centralex} is determined by its extension 
class $$[f]\in H^2(\Theta, \ZZ^k)  , $$ where  $f$ is a two-cocycle
with values in $\bbZ^{k}$. 
Recall that (by the Borel-density theorem) the lattice $\Theta$ does 
not contain any solvable or finite normal subgroup. 
Therefore, we remark: 
\begin{itemize} 
\item[(*)] the image of  $\ZZ^{k}$ in $\Gamma$ is the maximal virtually solvable 
normal subgroup of $\Gamma$. 
\end{itemize} \vspace{1ex}

\paragraph{\em First construction step}
Consider the standard inclusion $\ZZ^{k} \subset \bbR^{k}$, where $\ZZ^{k}$ is a lattice in
$\bbR^{k}$. Using the cocycle $f$ for \eqref{centralex},  we obtain
 a pushout diagram of group extensions, as follows: 
\begin{equation}\label{pushout} 
\begin{CD}
 1@>>> \ZZ^k@>>>\Gamma@>>> \Theta@>>>1\\
@.   \cap\,@. \cap@. ||@. \\
  1@>>> \RR^k @>>>\RR^k\cdot \Gamma @>>> \Theta@>>>1.\\
\end{CD}
\end{equation}

\smallskip 
\paragraph{\em Second construction step} 
The Seifert construction
shows that there exists a proper action of the pushout $\RR^k\cdot \Gamma$
on the product manifold 
\begin{equation} \label{product} X=\RR^k\times\HH^n \; .  \end{equation}  
This action extends the translation action of $\RR^k$ on the left factor of $X$,
and induces the original action of $\Theta$ on $\HH^n$. 
Thus the  quotient
$X/\Gamma$ is a compact aspherical orbifold and  
\begin{equation}  \label{seifertfib}
T^k \to  X/\Gamma \to  \HH^n/\Theta
\end{equation}
 a \emph{Seifert fibering} 
with typical fiber a $k$-torus $T^k$. See  \cite[Theorem 7.2.4, Sections 7.3, 7.4]{LeeRaymond}. 

\smallskip 
\paragraph{\em Third construction step}
The total space $X/\Gamma$ of the fibration \eqref{seifertfib} carries a compatible Riemannian metric of large local symmetry: 

\begin{proposition}[Metric of large symmetry on $X/\Gamma$]  \label{prop:metric_exists}
 There exists a Riemannian metric $g$ on $X$ such that  
 \begin{enumerate}
 \item $\RR^k\cdot \Gamma$ acts properly by isometries, and,
 \item   the projection map $X \to  \HH^n$ is a Riemannian submersion.
 \end{enumerate}
 \end{proposition}
\begin{proof} Since $\RR^{k}$, $\Theta$ and also  (according to 
 \cite[Theorem 7.2.4]{LeeRaymond}) $\Gamma$ act properly, we infer in the 
 light of the fact that  $\Gamma \cap \RR^{k}$ is lattice in $\RR^{k}$
 that the pushout $\RR^k\cdot \Gamma$ acts properly on $X$. Thus there
exists an $\RR^k\cdot \Gamma$-invariant Riemannian metric $g'$ on
$X$  (see \cite{Koszul}, for example). Let $\mathcal H $ be the 
horizontal distribution orthogonal to the orbits of $\RR^{k}$ with respect to this metric. 
If $\displaystyle \mathsf p:X\ra\HH^n$ is the projection onto the second factor of $X$, 
the induced bundle map ${\mathsf p}_*: \mathcal{H} \ra {T}\, \HH^n$ 
identifies horizontal spaces in $X$ with the respective tangent spaces of $ \HH^n$.

Put $u=T+A, v=S+B$ for $u,v\in T_{x} X$, $S,T$ tangent to the fibers, and 
$A, B \in \mathcal H$. 
Let $g_{\HH}$ denote the hyperbolic metric and define a Riemannian metric 
$g$ on $X$ by putting:  
\begin{equation}\label{isog} 
g(u,v)=g'(T,S)+g_{\HH}({\mathsf p}_*A,{\mathsf p}_*B) \; .
\end{equation} 
Therefore, the metric $g$ has the same horizontal spaces $\mathcal H$ as $g'$, and also  $\mathsf p: X \to \HH^n$ is a Riemannian submersion, by construction. 
Also $\RR^{k} \cdot \Gamma$  acts isometrically on the fibers, and the horizontal distribution $\mathcal{H}$  for $g$ is invariant by $\RR^{k} \cdot \Gamma$. Since  $\RR^{k}$ acts trivially on $\mathcal H$, the metric $g$ is clearly invariant by $\RR^{k}$. Moreover, since $\mathsf p$  is $\Gamma$-equivariant, and $\Theta$ acts by isometries on the base, also $\Gamma$ acts by isometries with respect to $g$. 
\end{proof}

\begin{corollary} \label{cor:large_symmetry}
The compact orbifold $X/\Gamma$ admits a metric, which has a
complete infrasolv tower of length one and of solvable rank $k$ (see Definition $\ref{def:solvrank}$). 
In particular, 
$X/\Gamma$ carries a metric of large symmetry.
\end{corollary}

Suppose $g$ is any Riemannian metric on $X$, which satisfies (1) and (2) of Proposition \ref{prop:metric_exists}. Then the situation is sufficiently rigid, so that the radical projection 
$$ \mathsf{q}: X \to X/\fR \; ,  $$
as defined in Section \ref{sect:rad2},  actually coincides with the projection map onto the second factor of \eqref{product}. More precisely:  

\begin{proposition}[Rigidity of the projection] \label{projection_rigid}
For any metric on $X$, satisfying $(1)$ and $(2)$ of Proposition $\ref{prop:metric_exists}$:
 \begin{enumerate}
 \item  The image of\/ $\RR^k$ in $\Isom(X)$ is the maximal simply connected normal solvable subgroup $\fR_{0}$ of\/  $\Isom(X)$. 
 \item The fibers of $\mathsf{q}$ are Euclidean spaces on which $\RR^{k}$  acts simply transitively by translations. 
  \item  The radical quotient $X/\fR$ is isometric to $\HH^n$, and the projection map onto the second factor of \eqref{product} corresponds to the radical projection 
 $\mathsf{q}$.
 \end{enumerate}
\end{proposition}
\begin{proof} 
Write ${\rm Isom}(X)^0=  \mathsf{G} = {\mathsf R}\cdot \mathsf{S}$, as
in Section \ref{sect:action_radicalquotient}. Put $\Gamma_{0} = \Gamma \cap \mathsf{G}$.  Then $\Gamma_{0}$ is a uniform lattice in $\mathsf{G}$. 
Consider the projection homomorphism $$ \phi: \Isom(X) \to \Isom(X/\fR) \; , $$ corresponding to the Riemannian submersion $\mathsf{q}$. 
If the image $$ S_0=\phi({\rm Isom}(X)^0) = \phi(\fS)  $$ is not the trivial group, it is a semisimple Lie group  of non-compact type and $\phi(\Gamma_{0})$ is a uniform lattice in $S_{0}$. (See $(7)$ of Theorem \ref{thm:structure}.) 

Denote with  $V$ the vector subgroup of $\Isom(X)$ 
which arises  by the isometric action of $\bbR^{k}$ on $X$.  
By construction,  $V$ is normalised by $\Gamma_{0}$, and the image $\phi(V)$ in $S_{0}$ is normalised by the uniform lattice $\phi(\Gamma_{0})$.  By Borel's density theorem (applied to the adjoint form of  $S_{0}$),  $S_{0}$ does not  contain any abelian connected subgroup normalised by $\phi(\Gamma_{0})$. This shows that $V$ is contained in $\ker \phi$. Consequently, $V$ acts by isometries on each fiber of the projection $\mathsf{q}$, and by construction the action is free. 

Next observe that the fibers of $\mathsf{q}: X \to X/\fR$ are at most $k$-dimensional. 
Indeed, by Proposition \ref{pro:Gam_kerphi},  $$\rad_{\fG}(\Gamma) = \Gamma_{0} \cap \ker \phi  $$ is a virtually polycyclic normal subgroup, and it acts with compact quotient on each fiber $\mathsf{q}^{-1}(y)$. This shows that
$\dim \mathsf{q}^{-1}(y) = \rank \rad_{\fG}(\Gamma)$, see, for example,
 the argument given in the proof of Corollary \ref{cor:solvrank}. 
As observed in (*), we must have $\rad_{\fG}(\Gamma) \subseteq \ZZ^{k}$, and consequently $\rank \rad_{\fG}(\Gamma) \leq k$. 

We deduce that  $V$ is a simply transitive group of isometries on each fiber $\mathsf{q}^{-1}(y)$. Since any left invariant metric on a  vector space is flat (and unique up to an  affine transformation), this implies that $\mathsf{q}^{-1}(y)$ is isometric to an Euclidean space $\EE^{k}$ on which $V$ acts by translations.  In particular, (2) holds. 

Recall from Proposition \ref{prop:kerphi} 
that $\ker \phi$ acts faithfully on $\mathsf{q}^{-1}(y)$ by isometries, so that 
$\ker \phi$ embeds as a subgroup of the Euclidean group $\Isom(\EE^{k})$.
It follows that the above translation group $V$ is normal in $\ker \phi$, and it is the nilpotent radical of $(\ker \phi)^0$. Recall that a maximal simply connected normal subgroup $\fR_{0}$ of $\fR$, containing the nilpotent radical of $\fR$, acts simply transitively on the fibers.  
We conclude that $V = \fR_{0}$. 
In particular (1), holds.  

So far it is shown that the Riemannian submersions $\mathsf{q}$ and the projection to 
$\HH^n$ have the same fibers. Clearly, any two Riemannian submersions $f_{i}: X \to B_{i}$ with the same fibers give rise to an isometry $h: B_{1} \to B_{2}$ such that $f_{2} = h \circ f_{1}$. Therefore, (3) holds. 
\end{proof}


Next we turn to describe the continuous symmetries of arbitrary $\Gamma$-invariant metrics on $X$. We will show that these are mainly determined by the group extension \eqref{centralex}. 

\subsection{Symmetry and rigidity of fiberings} \hspace{1cm}

\smallskip 
A Riemannian submersion $\mathsf{q}: X \to Y$ will be called \emph{$\Gamma$-compatible}, if
the following hold:  
$\Gamma \leq \Isom(X)$ permutes the fibers of $\mathsf{q}$, the image of the induced map $\phi: \Gamma \to \Isom(Y)$ is a discrete uniform subgroup, and the kernel of $\phi$ is virtually solvable.

\begin{proposition} \label{prop:adjointform}
Let $\mathsf{q}: X \to Y$ be a $\Gamma$-compatible Riemannian submersion, such
that\/ $Y$ is a homogeneous Riemannian space, with\/  $\Isom(Y)^{0}$ semisimple of non-compact type. Then the adjoint form of\/ $\Isom(Y)^{0}$ is isomorphic to $\PSO(n,1)$. 
\end{proposition} 
\begin{proof} 
Put $S = \Isom(Y)^{0}$,  $\Theta_{\phi} = \phi(\Gamma) \cap S$ and $\Gamma_{0} = \phi^{-1}(\Theta_{\phi})$.  Note that $\Gamma_{0}$ is a finite index characteristic subgroup in $\Gamma$.
It follows that  $\Lambda_{0} = \Gamma \cap \ZZ^{k}$ is the maximal solvable normal subgroup of $\Gamma_{0}$. In particular, $\Gamma_{0}$ satisfies an exact sequence $1\ra \Lambda_{0} \ra \Gamma_{0} \lra \Theta_{0} \ra 1 $, analogous to
\eqref{centralex}, 
where $\Theta_{0}$ is a uniform lattice in $\PSO(n,1)$.

Let $\hat \phi: \Gamma_{0} \to \hat S$ denote the induced map, where $\hat S$ is the 
adjoint form of $S$. Since the image  $\phi(\Gamma)$ divides $Y$, 
$\phi(\Gamma) \cap S$ is a uniform lattice in $S$, and 
$\hat \phi(\Gamma_{0})$ is a uniform lattice in $\hat S$.

Since $\ker \hat \phi$ is an extension of an abelian group (the center of $S$)  by $\ker \phi \cap \Gamma_{0}$, and the latter is virtually polycyclic, $\ker \hat \phi$ is virtually solvable.  By the remark in (*), 
the image of  $ \ker  \hat \phi$ 
in $\Theta_{0}$ must be trivial. This shows that $\ker \hat \phi$
is contained in the characteristic central subgroup $\Lambda_{0}$ of $\Gamma_{0}$.   
A symmetric argument with respect to the homomorphism $\hat \phi: \Gamma_{0} \to \hat \phi(\Gamma_{0})$ shows that, 
 in fact, $\ker \hat \phi = \Lambda_{0}$. We conclude that there exists an isomorphism of uniform lattices $$ \bar \phi:\Theta_{0 } \to \hat \phi(\Gamma_{0})\;  . $$ 

If $n \geq 3$, the Mostow strong rigidity theorem  \cite[Theorem A']{Mostow1}
states that $\bar \phi$ extends to an isomorphism $ \mathrm{PSO}(n,1) \ra \hat S_0$.  
In the case $n=2$, $\Theta_{0}$ is a surface group, and so is $\hat \phi(\Gamma_{0})$.
Therefore, $\hat S_0$ is isomorphic to ${\rm PSO}(2,1)$. 
\end{proof}

The symmetry properties of $\Gamma$-invariant metrics on $X$ are tightly coupled to the group extension \eqref{centralex}: 

\begin{theorem}[Isometry group and extension class] \label{thm:non_locally_homogeneous}
Let $\Theta$ be a torsion free lattice in  ${\rm PSO}(n,1)$. 
Assume that  the extension class for $\Gamma$, defined by \eqref{centralex},  has infinite order. 
Then  the following hold for the manifold $X/\Gamma$:
\begin{enumerate}

\item[(1)]  $X/\Gamma$  admits a metric of large symmetry.
\item[(2)]  For $n\geq 3$, $X/\Gamma$ does not admit a locally homogeneous Riemannian metric. 
\item[(3)]  For $n=2$, $X/\Gamma$ admits the structure of a locally homogeneous Riemannian manifold. In particular, $\Gamma$ embeds into a connected Lie group. 
\end{enumerate}
\end{theorem}
\begin{proof} Now (1) is just a special case of Corollary \ref{cor:large_symmetry}.

For (2), suppose that there exists a $\Gamma$-invariant homogeneous 
Riemannian metric on $X$.  In particular, $\fG= {\rm Isom}(X)^{0}$ acts transitively on $X$. 
By Theorem \ref{thm:structure}, 
the radical projection $\mathsf{q}: X \to Y$ is $\Gamma$-compatible and the image 
$S_{0} = \phi(\fG) \leq \Isom(Y)$ is semisimple of non-compact type, and normal in 
$\Isom(Y)^{0}$. Since $\Gamma$ is not solvable, it is clear that $Y$ and $S_{0}$ are non-trivial. Since $\fG$ acts transitively on $X$, $S_{0}$ acts transitively on $Y$. It is clear by now that $S_{0}$ is a Levi subgroup of $\Isom(Y)^{0}$.  If necessary
(see  Example \ref{ex:sl2R}), 
we may repeat the process of dividing out the radical.
 In any case, there exists a $\Gamma$-compatible Riemannian submersion $\mathsf{q}: X \to Y$, where $S_{0} = \phi(\fG) = \Isom(Y)^{0}$
is semisimple of non-compact type and acts transitively on $Y$. 
By Proposition \ref{prop:adjointform}, $S_{0}$  is locally isomorphic to $\PSO(n,1)$.
So is the non-compact type semisimple part $S$ of a Levi subgroup of $\fG$.

Now we are assuming  $n \geq 3$. Therefore $S$ is, in fact, a finite cover of $\PSO(n,1)$, 
and from the beginning the image $S_{0}$ of the radical projection in $\Isom(Y) = \Isom(X/\fR)$
is $\PSO(n,1)$.  As before, we write  $$ \fG ={\rm Isom}(X)^0={\mathsf R} \cdot \mathsf{S} = {\mathsf R} \, \mathcal{K}^0 S \, ,$$  where $ \mathcal K^0$ is compact semisimple. 
Also, by going down to a finite index subgroup, we may assume that $ \Gamma \subset {\rm Isom}(X)^0 $.
 The inclusion of $\Gamma$ then satisfies $\Gamma\cap ({\mathsf R}\cdot \mathcal K^0)= \ZZ^k$. 

Since the nilpotent radical $\fN$ of $\fR$ is simply connected and $S$ is a finite cover of 
${\rm PSO}(n,1)$, it is not difficult to see  
that (up to dividing a finite group in the center), 
$\fG$ is a linear group. 
Recall further 
that,  by Corollary \ref{cor:no_normal_compact}, 
the maximal compact normal semisimple subgroup of $\fG$ is trivial.
Moreover, by Theorem \ref{thm:structure} (2), the intersection of $\Gamma$ with
$\fN$  is a lattice in $\fN$. Since the extension \ref{centralex} is central, the Borel 
density theorem 
implies that $\fN$ is centralised by $S$. Hence, $S$ is, in fact, a  factor of $\fG$.
Therefore, the deformation Theorem  of Mostow Proposition \ref{prop:Mostow_straight} 
states that (a finite index subgroup of) the lattice $\Gamma$ can be deformed into 
a lattice $\Gamma'$, where $ \Gamma' = ( \Gamma' \cap R \cK^{0}) \cdot  (\Gamma' \cap S) $.


It follows that  a finite index subgroup of $\Gamma$ splits as a direct product with factor 
$\bbZ^{k}$. However, this is  impossible,  
since the extension class for $\Gamma$ is assumed to have infinite order. The contradiction shows that (2) holds. 

For (3) recall that any torsion-free lattice 
$\tilde \Gamma$ in $\widetilde \PSL(2,\bbR)$ admits an exact sequence
of the form $1\ra \ZZ \ra  \tilde \Gamma \ra \Theta \ra 1 $, where $\Theta$ is 
a lattice in $\PSL(2,\bbR)$. The extension class of the exact sequence, is determined by the index of $\tilde \Gamma \cap Z$, where
$Z$ is the center of  $\widetilde \PSL(2,\bbR)$.  
From another point of view, this identifies with the Euler-number of the associated circle 
bundle $$ \widetilde \PSL(2,\bbR) / \tilde \Gamma \to \PSL(2, \bbR)/\Theta \; . $$  Such bundles are studied in detail in  \cite[(8.5) Theorem (b)]{Kulkarni-Raymond}. 

Since $\HH^{2}/\Theta$ is an orientable surface, we have $$ H^{2}\left(\HH^{2}/\Theta \, , \, \ZZ \right) =  H^{2}\left(\Theta \, , \, \ZZ\right) \cong \ZZ \; . $$ In particular, in the case $k=1$, $n=2$, 
there exists for any extension class of infinite order  a locally homogeneous orbifold $X/ \Gamma$, where $X= \widetilde \PSL(2,\bbR)$ and 
$\Gamma$ embeds as a lattice in $\widetilde \PSL(2,\bbR)$. This proves (3) in the case $k=1$.  

For $k \geq 2$, we sketch the proof as follows: Note first that the homology functor $H^{2}(\Theta, \, \cdot \, )$ induces an 
isomorphism of abelian groups
$$ \Hom_{\ZZ}(\ZZ, \ZZ^{k}) \lra \Hom_{\ZZ} \! \left(H^{2}(\Theta, \ZZ) \, ,\,  H^{2}(\Theta, \ZZ^{k})\right) \; .
$$ 
This means, in particular, that, given any 
$\psi: H^{2}(\Theta, \ZZ) \to  H^{2}(\Theta, \ZZ^{k})$, there
exists $\varphi: \bbZ \to \ZZ^{k}$ such that $\psi = H^{2}(\Theta, \varphi)$.
Now let $[f] \in H^{2}(\Theta, \ZZ^{k})$ be an extension class. 
Since $H^{2}(\Theta, \ZZ) = \ZZ$, there exists $\psi$,  with $\psi(1) =  [f]$, and
a corresponding $\varphi: \bbZ \to \ZZ^{k}$. 

By this construction, there exists a pushout diagram 
\begin{equation}\label{pushout2} 
\begin{CD}
 1@>>> \ZZ @>>>\Gamma_{0}@>>> \Theta@>>>1\\
 @.  @VV\varphi V  \cap@. ||@. \\
  1@>>> \ZZ^k @>>> \Gamma @>>> \Theta@>>>1.\\
\end{CD} \; , 
\end{equation} 
such that the extension class of the bottom row is precisely $[f]$. Moreover, 
the pushout  diagram 
\begin{equation*}\label{pushout3} 
\begin{CD}
 1@>>> \ZZ @>>> \widetilde \PSL(2,\bbR)@>>> \PSL(2,\RR) @>>>1\\
@.  @VV\varphi V  \cap@. ||@. \\
  1@>>> \RR^k @>>> \RR^{k} \cdot_{\varphi(\ZZ)} \widetilde \PSL(2,\bbR) @>>>  \PSL(2,\RR) @>>>1 \\
\end{CD}  
\end{equation*}
shows that $\Gamma$ embeds as a discrete uniform subgroup of  the push out Lie group 
$$ \fG = \RR^{k} \cdot_{\varphi(\ZZ)} \widetilde \PSL(2,\bbR) \, .$$ 
\end{proof}

Groups which satisfy the assumption of Theorem $\ref{thm:non_locally_homogeneous}$ 
are obtained 
by finding a discrete
cocompact hyperbolic group $\Theta$ whose representative cocycle
$[f]\in H^2(\Theta,\ZZ^k)$ is of infinite order. For example, 
there exists a compact hyperbolic $3$-manifold 
$\HH^3/\Theta$,  whose Betti number $b_1$ is not zero in
$H_1(\Theta,\ZZ)$, see \cite{MI}.  As $H_1(\Theta, \ZZ )\otimes \ZZ^k=H_1(\Theta, \ZZ^k)\cong H^2(\Theta;\ZZ^k)$, the fundamental group $\Theta$ has representative cocycles of infinite order in 
$H^2(\Theta,\ZZ^k)$.

\begin{corollary} \label{cor:non_locally_homogeneous}
There exists a compact aspherical Riemannian manifold $X/\Gamma$ of dimension $4$ that admits a complete infrasolv tower of length one, 
which is fibering over a $3$-dimensional hyperbolic manifold.  
Moreover,  the manifold $X/\Gamma$ does not admit any locally homogeneous 
Riemannian metric.
\end{corollary}

\section{Application to tori}\label{sec:fake}

An $n$-dimensional \emph{exotic torus} $\tau$ is a compact smooth manifold
homeomorphic to the standard $n$-torus $T^n$ but not diffeomorphic
to $T^n$. In this section
we shall prove
that \emph{an exotic torus has no large symmetry} (compare Definition \ref{def:symmetry}). 
In the proof we replace the infrasolv tower  \eqref{eq:tower2} with its associated 
\emph{infranil tower}, which is obtained by naturally dissecting infrasolv orbibundles into 
a composition of infranil orbibundles. 

\subsection{Infrasolv orbibundle fiber over infranil orbibundles} 
Let  $$ \displaystyle \mathsf{p}: X/\Gamma \ra Y/\Theta$$ 
be an infrasolv fiber bundle with associated 
group extension 
\begin{equation}\label{phigroup}
\begin{CD} 1 \to \Delta \to  \Gamma \to  \Theta \to 1 , 
\end{CD}
\end{equation} 
as in Definition \ref{fibersolv}, where $Y = X/\fR$. The induced homomorphism
$$ \phi:{\rm Isom}(X)\ra {\rm Isom}(Y)$$ 
satisfies $\Delta=\ker \,\phi\cap \Gamma$. Recall that the bundle $\mathsf{p}$ is modeled on the
maximal simply connected normal subgroup $\fR_{0}$ of $\fR$, which acts freely on $X$. 
If ${\sf N}$ is the nilpotent radical of $\sf R={\sf R}_0{\sf T}$, then
it is a simply connected characteristic subgroup with  ${\sf N}\leq {\sf R}_0$.
As ${\sf N}$ acts freely on $X$ by Lemma \ref{lem:nilradical}, we may put
$$ Z=X/{\sf N},\ \ \, {\sf V}={\sf R}_0/{\sf N} \, , $$
where $Z$ carries the Riemannian quotient metric from $X$. 
Then the homomorphism $\phi$ factors over $\Isom(Z)$ as follows: 
\begin{equation}\label{twohomo}
\begin{CD}
&\Gamma\,\leq\,{\rm Isom}(X)@>\phi>>\Theta\,\leq\,{\rm Isom}(Y)\\
&  \searrow \phi_{1}  @.  \nearrow \phi_{2} @. \\
& &Q\,\leq\,{\Isom}(Z) \\
\end{CD}\end{equation}

\smallskip 
As $\Gamma\leq{\Diff}(X,{\sf R}_0)$, note
that $\phi_{1}(\Gamma)=Q\leq {\Diff}(Z,{\sf V})$.

\begin{pro}\label{Seidecomp}
The infrasolv fiber bundle
$\displaystyle \mathsf{p}: X/\Gamma \ra Y/\Theta$
decomposes into a composition of an infranil bundle and   an infra-abelian bundle: 
\begin{equation*}\label{twobundles}
\begin{CD}
X/\Gamma\,\ @>\mathsf{p}>>\,Y/\Theta\\
  \mathsf{p}_1\searrow   @.  \nearrow {\mathsf{p}}_2@. \\
&Z/Q \\
\end{CD} \; \; \; . \end{equation*}
\end{pro}

\begin{proof}
Let $\displaystyle {\sf N}\ra X\stackrel{\mathsf{q}_1}\lra Z$
be the  principal bundle associated to the action of $\fN$ on $X$. 
As $\Gamma\leq {\Diff}(X,{\sf R}_0)$, $\ker \phi_{1} \cap \Gamma$ normalizes 
$\fN$ and may be seen as a group of affine transformations of the fibers of $\mathsf{q}_{1}$.
This induces an embedding $$ \ker \phi_{1} \cap \Gamma \leq {\Aff}({\sf N})$$   such that 
$\ker \phi_{1} \cap \Gamma$ acts properly discontinuously on ${\sf N}$.
If we note that
${\sf N}\cap \Gamma$
is a uniform lattice in ${\sf N}$ from (2) of Theorem \ref{thm:structure},
then ${\sf N}\cap \Gamma$ is a finite index subgroup of $\ker \phi_{1} \cap \Gamma$.
Thus the quotient $$ {\sf N} \big/  \! \left(\ker \phi_{1} \cap \Gamma \right)$$  is an infranil orbifold.
Also $G={\sf N} \, \Gamma$ is  a closed subgroup of ${\rm Isom}(X)$. 
As ${\sf N}$ is normal in $G$, we apply Lemma \ref{lem:properquotient} to
deduce  that $$ G/ \, {\sf N}=\Gamma \big/ \, \Gamma\cap {\sf N}$$  acts properly (discontinuously)
on $X/{\sf N}=Z$. In particular,  $$Q=\Gamma \big/ \left( \ker  \phi_{1} \cap \Gamma \right)$$  acts properly discontinuously on $Z$ with compact quotient.   
Thus $$ {\sf p}_1 :X/\Gamma\,\ra\,Z/Q$$ is an infranil fiber bundle (compare the proof of 
Theorem \ref{thm:basic_bundle}). 
\smallskip

Next observe that  $\displaystyle \mathsf{p}_2 :Z/Q\ra Y/\Theta$ 
is a fiber bundle with ${\sf V}$-geometry in the sense of  Definition \ref{fibersolv}. 
By the commutative diagram \eqref{twohomo}, it is easy to see that
$\phi_{1}(\Delta)=\ker {\phi_{2}} \cap Q$.
In particular, since $\phi_{1}(\Delta)$ acts properly discontinuously on ${\sf V}$ by affine transformations, 
and the quotient ${\sf V}/\phi_{1}(\Delta)$ is a compact Hausdorff space. 

From (2) of Proposition \ref{prop:kerphi},  $(\ker \,\phi)^0= {\sf R}_0 ({\sf T}{\cK}^0)$.  
As $\Delta\leq \ker \,\phi$ and $\ker \,\phi$ has finitely many components,
$\Delta\cap(\ker \,\phi)^0$ is a finite index subgroup of $\Delta$.
On the other hand, if we apply 
Proposition \ref{prop:Kisinthekernel} to $\phi_{1}:{\rm Isom}(X)\ra {\rm Isom}(Z)$,
then it is noted that ${\sf T}{\cK}\leq \ker \,\phi_{1}$. 
Since  $\Delta\cap(\ker\phi)^0 \leq {\sf R}_0({\sf T}{\cK}^0)$
and $\phi_{1}({\sf R}_0) ={\sf V}$, it follows 
$$\phi_{1}(\Delta\cap(\ker \phi)^0)\leq{\sf V}$$
is a discrete uniform subgroup of $V$.
Hence ${\sf V}/\phi_{1}(\Delta)$ is a compact Euclidean orbifold.
As a consequence $\displaystyle {\mathsf{p}}_2 : Z/Q\ra Y/\Theta$ 
is an infra-abelian bundle.
This finishes the proof of Proposition \ref{Seidecomp}.
\end{proof}

\subsection{Exotic tori}\label{subsec:fake}

We come now to the proof of: 

\begin{theorem}\label{theorem:fake}
Let $\tau$ be an $n$-dimensional exotic torus. Then $\tau$ does not
admit any Riemannian metric of \emph{large symmetry}.
\end{theorem}

\begin{proof}
Put $\Gamma=\pi_1(\tau)\cong \ZZ^n$ and let $X$ be the universal covering space of $\tau$.
Suppose that $\tau$ has \emph{large symmetry}.
By the definition, there exists 
a tower of Riemannian submersions with infrasolv fibers 
\begin{equation}\label{alltower}
X/\Gamma \lra  X_1/\Gamma_1 \,\lra \,  \ldots \, \lra\, X_{\ell-1}/\Gamma_{\ell-1} 
\, \lra  \,
\{{\rm pt}\}. 
\end{equation}
As $\Gamma_i$ divides $X_i$, applying Lemma \ref{lem:finite_trivial} to 
each $\Gamma_i\leq{\rm Isom}(X_i)$ shows
that $\Gamma_i$ is \emph{torsion-free}, that is, $\Gamma_{i}$ is  a free abelian subgroup $(i=0,\dots,\ell)$.
Since each $\Gamma_i\cap{\rm Isom}(X_i)^0$ is a uniform abelian subgroup
in ${\rm Isom}(X_{i})^0$ by Lemma \ref{lem:isometry_lattices},
Proposition \ref{prop:key-correct} implies 
$$ {\rm Isom}(X_{i})^0={\sf R}_i \, {\sf K}_i \, , $$ where
${\sf R}_i$ is the solvable radical
and ${\sf K}_i$ is a compact connected semisimple group.
Since  \eqref {alltower} is an \emph{infrasolv tower}, 
${\rm Isom}(X_{i})^0$ contains a nontrivial nilpotent radical
${\sf N}_i$.
By Proposition \ref{Seidecomp}, each infrasolv bundle
$\displaystyle X_i/\Gamma_i \lra  X_{i+1}/\Gamma_{i+1}$
dissects  into infranil bundles:
\begin{equation}\label{Qinsert}
X_i/\Gamma_i \lra\,  Z_i/Q_{i}\lra\,  X_{i+1}/\Gamma_{i+1}.
\end{equation}
Since $Q_i$ acts properly discontinuously and $Z_i/Q_i$ is compact, 
note that ${\sf K}_i$ 
acts trivially on $Z_i=X_i/{\sf N}_i$
as in the proof of Proposition \ref{prop:Kisinthekernel}.

Inserting $Z_i/Q_i$ to the sequence  \eqref{alltower}
, we may assume  from the beginning that
the tower \eqref{alltower} is an 
infranil tower. 

With this assumption in place we put  $\Delta_i=\ker \phi_i\cap \Gamma_i$, where $\phi_{i}: \Isom(X_{i}) \to \Isom(X_{i+1})$
denotes  the natural map, $i=0,\dots,\ell-1$.
Since the tower is infranil,  $\Delta_i$ is contained in ${\rm Aff}({\sf N}_i)$, and 
the holonomy image of $\Delta_{i}$  has  compact closure in ${\rm Aut}({\sf N}_i)$.
On the other hand,
${\sf N}_i\cap \Gamma_i$ is a uniform lattice in ${\sf N}_i$
by $(2)$ of  the Theorem \ref{thm:structure}.
Since ${\sf N}_i\cap\Gamma_i$ is free abelian,
${\sf N}_i$ is isomorphic to the vector space $\RR^{n_i}$ for 
$n_i={\rm dim}\,{\sf N}_i$. Then $\Delta_i$
is a Bieberbach group in the Euclidean group ${\rm E}(\RR^{n_i})$. As $\Delta_i$ is free abelian,
we have $\Delta_i=\RR^{n_i}\cap \Gamma_i$.

As $\Gamma_i$ \emph{normalizes} $\RR^{n_i}$,
associated with the group extension $$\displaystyle 1\ra
\Delta_i\ra\Gamma_i\lra \Gamma_{i+1}\ra 1 \; , $$
there is an  \emph{injective Seifert fibering}  (see \cite{LeeRaymond}):
\begin{equation}\label{solvorbi}
\begin{CD}
\RR^{n_i}/\Delta_i@>>> X_i/\Gamma_i@>p_i>> X_{i+1}/ \Gamma_{i+1}
\end{CD}\end{equation}
where $\RR^{n_i}/\Delta_i$ is the standard $n_i$-torus. 

Since $X_{\ell-1}/\Gamma_{\ell-1}\ra X_{\ell}=\{{\rm pt}\}$ is also a Seifert fibering,
$\RR^{n_{\ell-1}}/\Delta_{n_{\ell-1}}=X_{\ell-1}/\Gamma_{\ell-1}$ is an $n_{\ell-1}$-torus.
Assume inductively from \eqref{solvorbi}
that $X_{1}/\Gamma_{1}$ is diffeomorphic to $T^{n_1}=\RR^{n_1}/\ZZ^{n_1}$.
Let $\bar\phi:
\Gamma_1\lra \ZZ^{n_1}$ be an isomorphism induced by
an equivariant diffemorphism $\bar\varphi:X_1\ra \RR^{n_1}$. 
Then we have an isomorphism $\phi: \Gamma \ra
\Delta_0\times \ZZ^{n_1}\cong\ZZ^{n_0+n_1}$ which makes the following
diagram commutative: 
\begin{equation}\label{isoSei}
\begin{CD}
1@>>>\Delta_0@>>>\Gamma @>>> \Gamma_1@>>>1\\
@. @V {\rm id}VV     @V {\phi}VV     @V{\bar\phi} VV \\
1@>>>\Delta_0 @>>> \Delta_0\times \ZZ^{n_1} @>>>\ZZ^{n_1}@>>>1.
\end{CD}
\end{equation} By the \emph{Lee-Raymond-Seifert rigidity} for abelian fiber 
\cite{LeeRaymond},
there exists a fiber-preserving equivariant diffeomorphism
$$(\phi,\varphi):(\Gamma,X)\lra (\ZZ^{n},\RR^{n}) \; . $$
Hence,  $X/\Gamma$ is diffeomorphic to the torus $T^{n}$
$(n=n_0+n_1)$. 
This proves the induction step and finishes the proof.
\end{proof}

\end{document}